\documentclass[10pt]{article}
\usepackage{amsfonts}
\usepackage{dsfont}
\usepackage{amsmath,mathrsfs}
\usepackage{amsthm}
\usepackage{amssymb}
\usepackage{amscd}
\usepackage{graphicx}
\usepackage{float}
\usepackage{subfigure}
\usepackage{epstopdf}
\usepackage{appendix}
\usepackage{booktabs}
\usepackage{units}
\usepackage{enumerate}
\allowdisplaybreaks[1]
\newtheorem{definition}{Definition}[section]
\newtheorem{theorem}{Theorem}[section]

\newtheorem{lemma}{Lemma}[section]
\newtheorem{proposition}{Proposition}[section]
\newtheorem{remark}{Remark}[section]
\newtheorem{example}{Example}[section]
\newtheorem{assumption}{Assumption}[section]
\newtheorem{problem}{Problem}[section]

\begin{document}
\title{{Linear-Quadratic Graphon Mean Field Games with Common Noise}\thanks{De-xuan Xu, Zhun Gou, and Nan-jing Huang: The work of these authors was supported by the National Natural Science Foundation of China (12171339).}}
\author{{De-xuan Xu$^a$, Zhun Gou$^b$, Nan-jing Huang$^a$\footnote{Corresponding author,  E-mail: nanjinghuang@hotmail.com; njhuang@scu.edu.cn} and Shuang Gao$^c$ }\\
{\scriptsize\it a. Department of Mathematics, Sichuan University, Chengdu,
Sichuan 610064, P.R. China}\\
{\scriptsize\it b. School of Mathematics and Statistics, Chongqing Technology and Business University, Chongqing 400067, P.R. China}\\
{\scriptsize\it c. Department of Electrical Engineering, Polytechnique Montreal-University of Montreal, Montreal, Canada}}
\date{}
\maketitle
\vspace*{-9mm}
\begin{center}
\begin{minipage}{5.8in}
{\bf Abstract.}
This paper studies linear quadratic graphon mean field games (LQ-GMFGs) with common noise, in which a large number of agents are coupled via a weighted undirected graph. One special feature, compared with the well-studied graphon mean field games, is that the states of agents are described by the dynamic systems with the idiosyncratic noises and common noise. The limit LQ-GMFGs with common noise are formulated based on the assumption that these graphs lie in a sequence converging to a limit graphon. By applying the spectral decomposition method, the existence of Nash equilibrium for the formulated limit LQ-GMFGs is derived. Moreover, based on the adequate convergence assumptions, a set of $\epsilon$-Nash equilibrium strategies for the finite large population problem is constructed. Finally, an application is given for network security to illustrate our theoretical results.
\\ \ \\
\end{minipage}
\end{center}
\section{Introduction}

The stochastic large population game problems have attracted considerable attention due to their wide applicability in a variety of areas, such as engineering, economics, finance and management science. For the large population with homogeneous interactions, Huang, Caines, and Malham\'{e} \cite{Huang,Huang1} and Lasry and Lions \cite{Lions} introduced the theory of mean field games (MFGs) independently, which has been extensively researched and developed rapidly in recent decades (see for instance \cite{Carmona,Carmona1,Carmona2,Bensoussan,Ma,Xu,Delarue17}).

However,  despite a few MFG models incorporating heterogeneity in individual characteristics \cite{Delarue17}, the MFG theory remains largely confined to games with homogeneous interactions. The complexity of the underlying network couplings makes such problems challenging and even intractable by standard methods. To characterize large graphs and to analyze the convergence of graph sequences to their limits, the graphon theory was established in \cite{Lovasz,Lovasz1,Borgs,Borgs1}. Gao and Caines \cite{Gao1,Gao2,Gao3,Gao4} proposed the theory of the graphon control to obtain the approximate optimal control for the complex and large size network systems, by applying the graphon theory and the infinite dimensional analysis theory.
Recently, based on the graphon theory, the strategic decision problems with heterogeneous interactions have been widely researched. Graphon static games were studied in \cite{Carmona3,Parise}. Caines and Huang proposed graphon mean field game (GMFG) theory in \cite{Caines1,Caines2,Caines} in which GMFG and the GMFG equations have been formulated for the analysis of dynamic games on a large non-uniform network, and the unique existence of solution for the GMFG equations was obtained.  We refer to \cite{Gao,Gao5} for linear quadratic GMFGs with deterministic and stochastic dynamics respectively, in which the analyses are both based on spectral decompositions. Aurell et al. \cite{Aurell} discussed a class of linear quadratic GMFGs set in a Fubini extension of a product probability space to overcome the joint measurability problem of the agent state trajectories with respect to labels and samples. Tchuendom et al. \cite{Foguen-Tchuendom} studied linear quadratic stochastic GMFGs with infinite horizon and proposed a sufficient and necessary condition under which a particular node in the network is associated with minimal equilibrium cost. Amini et al. \cite{Amini} studied stochastic GMFGs with jumps induced by a Poisson random measure. For more inspiring elaboration on graphon theory and graphon mean field games, one can refer to \cite{Bayraktar,Bayraktar1,Amini1,Aurell1}.

On the other hand, in addition to endogenous noise, individual agents may also be affected by exogenous noise, named ``common noise'', in many real situations. As pointed out by Bensoussan et al. \cite{Bensoussan1},  the common noise plays a significant role in modelling the dynamic behaviors of agents when taking into uncertainty features in games. In fact, the common noise can be interpreted as some exogenous factors, such as interest rate, exchange rate, price of raw materials and public policy of the government, which can affect all participants in the market with large population. There exist some works about mean field games with common noise. Carmona et al. \cite{Carmona4} developed the theory of existence and uniqueness for general stochastic differential mean field games with common noise. Tchuendom \cite{Foguen-Tchuendom1} showed that a common noise can restore the uniqueness property of Nash equilibrium in a class of mean field games. Other important researches addressing the mean field games with common noise can be found in the literature \cite{Carmona5,Bensoussan1,Graber,Wang,Hua}. Recently, for the large population with heterogeneous agents, Dunyak and Caines \cite{Dunyak} studied graphon mean field games in discrete time with Q-noise which is a generalized version of common noise, and Tangpi and Zhou \cite{Tangpi} discussed a graphon utility maximisation game with common noise. It is worth noting that in \cite{Tangpi}, the state equations do not contain the graphon mean field term, and the cost functions only consider the terminal cost. To our best knowledge, there is no other work to study graphon mean field games with common noise in continuous time.

The present paper is devoted to establish an $\epsilon$-Nash equilibrium for the finite large population game descibed by a class of LQ-GMFGs with common noise in continuous time, in which the states of all agents on a large non-uniform network are governed by the dynamic systems with the idiosyncratic noises and common noise.  In the case without common noise, such a problem has been considered in \cite{Gao5}. It should be noticed that our model can be used to capture the heterogeneity of all agents' interactions which can not be characterized by the model of \cite{Foguen-Tchuendom1}.  We would like to mention that the LQ-GMFGs with common noise provide a powerful tool for solving some real problems such as the strategic decision problems for a large number of competitive firms or agents affected by a common noise on a non-uniform network (see Section 5 for more details).

In the present paper, we would like to establish an $\epsilon$-Nash equilibrium for the finite large population game under some mild assumptions. Clearly, the methods used in those aforementioned papers do not apply here in a straightforward manner. Instead, one needs to carefully treat the conditional expectations of stochastic processes with three variables and the strong sectional information of the difference between step function type graphon corresponding to the finite graph and the limit graphon. By utilising the spectral decomposition method, Gronwall's inequality and the BDG inequality, we are able to establish an $\epsilon$-Nash equilibrium for the finite large population game.
The main contributions of this paper are threefold. Firstly, we consider a limit problem and obtain a Nash equilibrium solution depending on a set of forward backward stochastic differential equations (FBSDEs) (see Equation \eqref{FBSDEl}).  Secondly, we construct a set of $\epsilon$-Nash equilibrium strategies for the finite population problem from the solutions of the limit LQ-GMFGs, based on the convergence of the finite graphs to the limit graphon.
Finally, we apply the theoretical results to network security and obtain an approximate Nash equilibrium.

The organization of this paper is as follows. Section 2 is devoted to the background for graphons and other necessary
notations. In Section 3, we introduce the finite population problems and the corresponding limit graphon mean field game problems. Moreover, we establish the existence of Nash equilibrium for the limit situation. Section 4 studies the $\epsilon$-Nash equilibrium property for the strategies constructed from the solutions of the limit LQ-GMFG problems. In Section 5 we give an application to network security, before we summarize the results in Section 6.

To facilitate the calculation, we only consider the one-dimensional case in this work. The extension for the multi-dimensional case is straightforward.

\hspace*{\fill}\\
\noindent
\emph{Notation}: For $I=[0,1]$, the measure space over $I$ is denoted by $(I,\mathcal{B}_I,\lambda_I)$, where $B_I$ is the Borel $\sigma$-algebra and $\lambda_I$ is the Lebesgue measure. $L^p[0,1]$ $(p=1,2)$ denotes the space of equivalence classes of functions $\phi:\ [0,1]\rightarrow \mathbb{R}$ under the norm defined by $\|\phi\|_p=\left(\int_0^1|\phi(\alpha)|^p d\alpha\right)^{\nicefrac{1}{p}}$. The inner product in $L^2[0,1]$ is defined as follows: for $\phi,\varphi\in L^2[0,1]$, $\langle \phi,\varphi \rangle=\int_0^1\phi(\alpha)\varphi(\alpha)d\alpha$. The function $\mathds{1}\in L^2[0,1]$ is defined as follows: $\mathds{1}(\alpha)\triangleq 1$, for all $\alpha\in[0,1]$. $\|\cdot\|_\infty$ denotes the standard infinity norm for vectors and matrices.

\section{Preliminaries}
\qquad We first recall the concepts of graphs, graphons and graphon operators used in \cite{Gao4,Gao5,Lovasz}.

A graph $G=(V,E)$ is represented by a node set $V=\{1,\ldots,N\}$ and an edge set $E\subset V\times V$. The corresponding adjacency matrix is defined as $M_N=[m_{ij}]$, where $m_{ij}\in[-1,1]$ denotes the weight between nodes $i$ and $j$, so as to include graphs with possibly negative weights. A graph is undirected if its adjacency matrix is symmetric.

Generally, a graphon is defined as a symmetric measurable function from $[0,1]^2$ to $[0,1]$. We note that in some works, for instance \cite{Gao4}, the term ``graphon'' refers to a symmetric measurable function $M:[0,1]^2\rightarrow[-1,1]$. In this paper, we adopt the same definition of graphon in \cite{Gao4}. Let $\mathcal{W}_1$ denote the set of all graphons. The cut norm of a graphon is defined as follows:
$$
\|M\|_{\Box}=\sup\limits_{S,T\in\mathcal{B}_I}\left|\int\limits_{S\times T}M(x,y)dxdy\right|,
$$
which will be used in the studies for convergence of the sequence of graphons in Appendix. It's noted that for functions defined on $[0,1]$ or $[0,1]^2$, it usually doesn't matter whether we use Borel or Lebesgue $\sigma$-algebra \cite{Janson}. In this paper, we assume $[0,1]^2$ are equipped with the product $\sigma$-algebra $\mathcal{B}_I\times \mathcal{B}_I$ and the product measure $\lambda_I\times\lambda_I$, which is appeared in \cite{Janson}.

Specially, for the $N$-uniform partition $\{P_1,\ldots,P_N\}$ of $[0,1]$: $P_l=[\frac{l-1}{N},\frac{l}{N})$ for $1\leq l\leq N-1$ and $P_N=[\frac{N-1}{N},1]$, the step function type graphon $M^{[N]}$ corresponding to $M_N$ is given by
$$
M^{[N]}(\alpha, \beta)=\sum\limits_{q=1}^N \sum\limits_{l=1}^N \mathds{1}_{P_q}(\alpha)\mathds{1}_{P_l}(\beta)m_{ql},\qquad \forall (\alpha,\beta)\in [0, 1]^2,
$$
where $\mathds{1}_{P_q}(\cdot)$ is the indicator function.

A graphon operator $M:L^2[0,1]\to L^2[0,1]$ is defined as follows
\begin{equation}\label{graphonoperator}
(M\varphi)(\alpha)=\int_{[0,1]}M(\alpha,\beta)\varphi(\beta)d\beta,\quad \forall\varphi\in L^2[0,1],
\end{equation}
where $M(\cdot,\cdot)$ is a graphon. The graphon operator $M$ is self-adjoint and compact \cite{Lovasz}.

Let $(\Omega, \mathcal{F}, \mathcal{F}_t, \mathbb{P})$ be the complete probability space satisfying the usual hypothesis with $t\in[0,T]$. We denote by $L_{\mathcal{F}}^2(\Omega;L^2([0,T];\mathbb{R}))$ the set of all $\{\mathcal{F}_t\}_{t\in[0,T]}$-progressively measurable processes $X(\cdot)$ taking values in $\mathbb{R}$, such that
$$
E\int_0^T|X_t|^2dt<\infty.
$$
The notation $L_{\mathcal{F}}^2(0,T;\mathbb{R})$ is used for simplicity, when there is no confusion. Moreover, we denote by $L_{\mathcal{F}}^2(\Omega;C([0,T],\mathbb{R}))$ the set of all continuous, $\{\mathcal{F}_t\}_{t\in[0,T]}$-progressively measurable processes $X(\cdot)$ taking values in $\mathbb{R}$, such that
$$
E\sup_{0\leq t\leq T}|X_t|^2<\infty.
$$

\section{LQ-GMFGs with common noise}
\qquad In this section, we introduce the model of linear quadratic graphon mean field game with common noise on a weighted undirected graph, and obtain the existence for solution of the limit problem by applying spectral decomposition method.

\subsection{Finite population problems}

\qquad  Graphon mean field games, distributed over a weighted undirected graph with $N$-node represented by its adjacency matrix $M_N=[m_{ql}]$, are asymptotic versions of finite large population games. Each node is associated with a homogeneous group of agents, in which each agent is influenced by the state average across the its nodal population and the state averages across other nodal populations. Let $\mathcal{V}_c$ be the set of nodes,  $\mathcal{C}_l$  the population in the $l$th node, and $K=\sum_{l=1}^N|\mathcal{C}_l|$  the number of all individual agents.

Consider a finite horizon $[0, T]$ for a fixed $T>0$. Assume that $(\Omega, \mathcal{F}, \mathbb{P})$ is a complete probability space, in which a standard $(K+1)$-dimensional Brownian motion $\{W^0_t, w^i_t,\ 1\leq i\leq K\}$ is defined. Specially, $W^0$ is a common noise (i.e., the same $W^0$ for all agents) and $w^i$ is an individual noise for the agent $\mathcal{A}_i$ ($1\le i\le K$). Let $\mathcal{F}_t^{W^0}\triangleq\sigma\{W_s^0, 0\leq s\leq t\}\bigvee\mathcal{N}$ and $\mathcal{F}_t\triangleq\sigma\{ W_s^0, w_s^i, x_0^i;0\leq s\leq t, 1\leq i\leq K\}\bigvee\mathcal{N}$, where $\mathcal{N}$ is the set of all $P$-null sets and $x_0^i$ is the initial state of the agent $\mathcal{A}_i$.

For each $i\in\{1,\ldots,K\}$, the admissible control set $\mathcal{U}$ of agent $\mathcal{A}_i$  is defined to be a collection of $\{\mathcal{F}_t\}_{t\in[0,T]}$-progressively measurable process with $E\int_0^T|u_s|^2ds<\infty$ and the state of $\mathcal{A}_i$ is described according to the following dynamic
\begin{equation}\label{1}
  dx_t^i=(Ax_t^i+Bu_t^i+Dz_t^i)dt+\Sigma dw_t^i+\Sigma_0dW_t^0,
\end{equation}
where $A, B, D, \Sigma$ and $\Sigma_0$ are given constants, $x_t^i$, $u_t^i$ and $z_t^i$ are respectively the state, the control and the network state average. For each $\mathcal{A}_i$ in $\mathcal{C}_q$, $x_0^i$ is identically distributed with $|Ex_0^i|\leq C_\mu$ and Var$(x_0^i)\leq C_\sigma$, where $C_\mu$ and $C_\sigma$ are independent of $q$. For each $\mathcal{A}_i$ in $\mathcal{C}_q$, denote $\mu_q\triangleq Ex_0^i$. Assume that the initial states $\{x_0^i,1\leq i\leq K\}$ are independent and are also independent of Brownian motion $\{W^0, w^1,\ldots,w^K\}$. For each agent $\mathcal{A}_i$ in $\mathcal{C}_q$, $q\in\{1,\ldots,N\}$, the network state average is given by
\begin{equation}\label{meanfielditem}
  z_t^i=\frac{1}{N}\sum\limits_{l=1}^N m_{ql}\frac{1}{|\mathcal{C}_l|}\sum\limits_{j\in\mathcal{C}_l}x_t^j.
\end{equation}

For any given strategy $u^i\in\mathcal{U}$, it is easy to show that the state dynamic (\ref{1}) has a unique solution $x^i\in L_\mathcal{F}^2 (\Omega;C([0,T];\mathbb{R}))$, $i=1,\ldots,K$. In fact, let $X\triangleq(x^1,\ldots,x^K)^T$, $U\triangleq(u^1,\ldots,u^K)^T$, $W\triangleq(W^0,w^1,\ldots,w^K)^T$, $\mathbb{F}\triangleq [m_{ql}\frac{1}{|\mathcal{C}_l|}\mathbb{I}_{|\mathcal{C}_q|\times|\mathcal{C}_l|}]$ with
$$
\mathbb{I}_{|\mathcal{C}_q|\times|\mathcal{C}_l|}\triangleq
\begin{pmatrix}
1      & \cdots & 1     \\
\vdots & \ddots & \vdots\\
1      & \cdots & 1
\end{pmatrix},
\quad \mbox{and} \quad
\widetilde{\Sigma}\triangleq
\begin{pmatrix}
\Sigma_0 & \Sigma & &    \\
\vdots & & \ddots & \\
\Sigma_0 & & & \Sigma
\end{pmatrix}.
$$
Then we can rewrite the state dynamic (\ref{1}) as the following vector-valued SDE
$$
dX_t=\left[AX_t+BU_t+\frac{D}{N}\mathbb{F}X_t\right]dt+\widetilde{\Sigma}dW_t,
$$
which has a unique solution $X\in L_\mathcal{F}^2 (\Omega;C([0,T];\mathbb{R}^K))$ by the well known result for standard SDE (see \cite{lvqi}).

Denote by $u=(u^1,\ldots,u^K)$ the strategies of all $K$ agents, and $u^{-i}=(u^1,\ldots,u^{i-1},u^{i+1},\ldots,u^K)$ the strategies except the agent $\mathcal{A}_i$, $i=1,\dots,K$. The cost function of the individual agent $\mathcal{A}_i$ is defined by
\begin{equation}\label{2}
  J_i(u^i,u^{-i})=E\left\{\int_0^T\left[Q(x_t^i-\nu_t^i)^2+R(u_t^i)^2\right]dt+Q_T(x_T^i-\nu_T^i)^2\right\},\quad i=1,\dots,K,
\end{equation}
where $Q,Q_T\geq 0,\ R>0$, and $\nu_t^i\triangleq H(z_t^i+\eta)$ with $H,\eta\in \mathbb{R}$.

Suppose that the objective of each individual agent is to minimize her own cost function by properly controlling her own state dynamics. Then, we formulate the following dynamic optimization problem of the large population system.
\begin{problem}\label{Knash}
Find $u^i\in\mathcal{U}$, $i=1,\ldots,K$, such that
$$
J_i(u^i,u^{-i})\leq J_i(v^i,u^{-i}),\quad i=1,\dots,K
$$
holds for any admissible control $v^i\in\mathcal{U}$. The strategy $K$-tuple $u=(u^1,\ldots,u^K)$ is called a Nash equilibrium for the $K$-player game where each individual agent is minimizing the cost in (\ref{2}) subject to the SDE dynamics (\ref{1}).
\end{problem}

Similar to \cite{Carmona5}, directly finding Nash equilibria for Problem \ref{Knash} is feasible. However, the computational complexity of calibration of a Nash equilibrium (if it exists) is high, especially for large values of number of agents \cite{Bensoussan}. Thus, a convenient approach to obtain an approximation of the Nash equilibrium strategy is usually demanded. In the following sections, we first consider the Nash equilibrium of the limit graphon mean-field game, and then construct an asymptotic Nash equilibrium for Problem \ref{Knash} from the corresponding limit GMFG solutions. It is worth noting that, under a finite-rank assumption (see below) on the limit graphon, we only need to solve $d$  forward-backward differential equations ($d\ll N$) to obtain the Nash equilibrium for the limit GMFG with a continuum of agents. This approach allows us to derive the asymptotic Nash equilibrium for Problem \ref{Knash}, and significantly reduces the computational complexity compared to direct computation.

\subsection{Limit graphon mean field game problems}
\qquad The graphon mean field games approach provided in \cite{Caines} employs the idea of finding approximate Nash equilibrium based on the double limits $N\rightarrow\infty$ and $\min_{1\leq l\leq N}|\mathcal{C}_l|\rightarrow\infty$.

Based on the works \cite{Caines,Gao5,Foguen-Tchuendom}, the graphon mean field games with common noise can be stated as follows. Let $[0,1]$ be the index set of nodes $\alpha$ in the limit network. With the network interaction within a cluster being uniform, assume that for each node $\alpha\in[0,1]$, there exists a representative (or generic) agent, denoted $\mathcal{A}_\alpha$ whose state's dynamic is given by
\begin{equation}\label{xalpha}
  dx_t^\alpha=(Ax_t^\alpha+Bu_t^\alpha+Dz_t^\alpha)dt+\Sigma dw_t^\alpha+\Sigma_0dW_t^0,
\end{equation}
where $A, B, D, \Sigma$ and $\Sigma_0$ are the same as in \eqref{1}, $w_t^\alpha$ and $W_t^0$ are standard Brownian motions, $x_0^\alpha$, $w^\alpha$ and $W^0$ are mutually independent for each $\alpha\in[0,1]$, $|Ex_0^\alpha|\leq C_\mu$ and $\mbox{Var}(x_0^\alpha)\leq C_\sigma$ for each $\alpha \in [0,1]$,
and the graphon mean field, denoted by $z_t^\alpha$, is given by
$$
z_t^\alpha=\int_0^1M(\alpha,\beta)E[x_t^\beta|\mathcal{F}_T^{W^0}]d\beta,
$$
where $M$ is a given graphon on $[0,1]^2$.

Let $\mathcal{F}_t^\alpha\triangleq\sigma\{W_s^0, w_s^\alpha, x_0^\alpha;0\leq s\leq t\}\bigvee\mathcal{N}$. The admissible strategy set $\mathcal{U}^\alpha$ of agent $\mathcal{A}_\alpha$ is defined to be a collection of $(\mathcal{F}_t^\alpha)_{t\in[0,T]}$-progressively measurable process such that $E\int_0^T|u_s|^2ds<\infty$. A representative agent $\mathcal{A}_\alpha$ aims to minimize cost function given by
\begin{equation}\label{Jalpha}
  J(u^\alpha,z^\alpha)=E\left\{\int_0^T\left[Q(x_t^\alpha-\nu_t^\alpha)^2+R(u_t^\alpha)^2\right]dt
  +Q_T(x_T^\alpha-\nu_T^\alpha)^2\right\}, \quad \alpha \in [0,1],
\end{equation}
where $\nu^\alpha\triangleq H(z^\alpha+\eta)$.

In this section, for some families of stochastic processes, their corresponding indices are required to satisfy some measurability and integrability conditions. Such families can also be regarded as three-variables functions of index, time and sample. We now define two sets of such functions (families of stochastic processes).

We denote by $\mathcal{L}_c$ the set of functions $\phi:\ [0,1]\times[0,T]\times\Omega\rightarrow\mathbb{R}$, such that (i) for a.s. $\omega\in\Omega$, each $t\in[0,T]$,\quad $\phi(\cdot,t,\omega)\in L^1[0,1]$; (ii)
 for each $\alpha\in[0,1]$,\quad $\phi(\alpha,\cdot,\cdot)\in L_{\mathcal{F}^{W^0}}^2(\Omega;C([0,T];\mathbb{R}))$; (iii) for a.s. $\omega\in\Omega$, $\phi(\cdot,\cdot,\omega)$ is continuous in $[0,T]$ and measurable in $[0,1]$; (iv)
 for a.s. $\omega\in\Omega$,\quad $\int_0^1\int_0^T|\phi(\alpha,t,\omega)|dtd\alpha<\infty$, and for a.s. $\omega\in\Omega$, each $\alpha\in[0,1]$,\quad $\int_0^T|\phi(\alpha,t,\omega)|dt<\infty$.
Moreover, we denote by $\mathcal{L}_L$ the set of functions $\phi:\ [0,1]\times[0,T]\times\Omega\rightarrow\mathbb{R}$, such that for a.e. $\alpha\in[0,1]$,\quad $\phi(\alpha,\cdot,\cdot)\in L_{\mathcal{F}^{W^0}}^2(0,T;\mathbb{R})$.

Hereafter, for a function $\psi: [0,1]\rightarrow \mathbb{R}$, we will often use the notation $\psi^{\alpha}=\psi(\alpha)$ for all $\alpha \in [0,1]$. Meanwhile, for a function $\phi:\ [0,1]\times[0,T]\times\Omega\rightarrow\mathbb{R}$, we will often denote $\phi_t^\alpha=\phi(\alpha,t,\omega)$ for all $\alpha\in[0,1]$, $t\in[0,T]$ and $\omega\in\Omega$, or just use $\phi$ to denote $\phi(\cdot,\cdot,\cdot)$ when no confusion is possible.

The LQ-GMFG with common noise associated with (\ref{xalpha}) and (\ref{Jalpha}) is characterized as follows:
\begin{itemize}
\item[1)] (graphon mean field inputs) Fix a three-variables function $z\in \mathcal{L}_c$.
\item[2)] (control problems) Solve the following stochastic optimal control problem. For a.e. $\alpha\in[0,1]$, find optimal control $u^{\alpha,o}$ such that
    \begin{equation}\label{Jalphaoptimal}
  J(u^{\alpha,o},z^\alpha)=\inf\limits_{u^\alpha\in\mathcal{U}^\alpha}J(u^\alpha,z^\alpha)=\inf\limits_{u^\alpha\in\mathcal{U}^\alpha}E\left\{\int_0^T\left[Q(x_t^\alpha-\nu_t^\alpha)^2+R(u_t^\alpha)^2\right]dt+Q_T(x_T^\alpha-\nu_T^\alpha)^2\right\}
\end{equation}
subject to
\begin{equation}\label{xalphaoptimal}
  dx_t^\alpha=(Ax_t^\alpha+Bu_t^\alpha+Dz_t^\alpha)dt+\Sigma dw_t^\alpha+\Sigma_0dW_t^0.
\end{equation}
\item[3)](consistency conditions) Find $z$ such that, for a.e. $\alpha\in[0,1]$, a.s. $\omega\in\Omega$, each $t\in [0,T]$,
\begin{equation}\label{consistency}
  z_t^\alpha=\int_0^1M(\alpha,\beta)\varphi_t^\beta d\beta,
\end{equation}
where for a.e. $\beta\in[0,1]$, each $t\in [0,T]$, $\varphi_t^\beta$ is one of the conditional expectation of $x_t^{\beta,o}$ given $\mathcal{F}_T^{W^0}$, i.e. $E[x_t^{\beta,o}|\mathcal{F}_T^{W^0}]=\varphi_t^\beta$, $\mathbb{P}$ a.s..
\end{itemize}

\begin{remark}
Inspired by \cite{Foguen-Tchuendom1}, it can be shown that for all $t\in[0,T]$, $E[x_t^{\beta,o}|\mathcal{F}_T^{W^0}]=E[x_t^{\beta,o}|\mathcal{F}_t^{W^0}]$. Indeed, we denote by $\mathcal{F}_{t,T}^{W^0}$ the filtration generated by the increments of $W^0$ on $(t,T]$ augmented with $P$-null sets. Then we have $\mathcal{F}_T^{W^0}=\mathcal{F}_t^{W^0}\bigvee\mathcal{F}_{t,T}^{W^0}$. We note that $\mathcal{F}_{t,T}^{W^0}$ is independent of $\mathcal{F}_t^{W^0}$ and $\mathcal{F}_t^\beta$. Since $\sigma (x_t^{\beta,o})\subset\mathcal{F}_t^\beta$, $\mathcal{F}_{t,T}^{W^0}$ is independent of $\sigma (x_t^{\beta,o})$, and so $\mathcal{F}_{t,T}^{W^0}$ is independent of $x_t^{\beta,o}$. Thus, for all $t\in[0,T]$, we have $E[x_t^{\beta,o}|\mathcal{F}_t^{W^0}\bigvee\mathcal{F}_{t,T}^{W^0}]=E[x_t^{\beta,o}|\mathcal{F}_t^{W^0}]$.
\end{remark}

\begin{proposition}\label{proposition1}
 For any given process $z\in \mathcal{L}_c$, assume that there is a pair $(g,q^1)\in\mathcal{L}_c\times\mathcal{L}_L$ such that for $a.e.\ \alpha\in[0,1]$, $(g^\alpha,q^{1\alpha})$ satisfies for any $t\in[0,T]$, $a.s.\ \omega$
\begin{equation}\label{galpha}
\left\{
\begin{split}
   & dg_t^\alpha=\left[\left(\frac{B^2}{2R}f_t-A\right)g_t^\alpha+(2QH-Df_t)z_t^\alpha+2QH\eta\right]dt+q_t^{1\alpha}dW_t^0, \\
    & g_T^\alpha=-2Q_TH(z_T^\alpha+\eta).
\end{split}
\right.
\end{equation}
Then, for $a.e.\ \alpha\in [0,1]$, there exists an optimal control for the optimal control problems (\ref{Jalphaoptimal}-\ref{xalphaoptimal}) with the form
\begin{equation}\label{controlalpha}
  u_t^{\alpha,o}=-\frac{B}{2R}(f_tx_t^{\alpha,o}+g_t^\alpha),
\end{equation}
where $f$ and the optimal state process $x^{\alpha,o}$ is given by
\begin{equation}\label{optimalx}
\left\{\begin{split}
  &\dot{f}_t-\frac{B^2}{2R}f_t^2+2Af_t+2Q=0,\ f_T=2Q_T,\\
  &dx_t^{\alpha,o}=\left[\left(A-\frac{B^2}{2R}f_t\right)x_t^{\alpha,o}-\frac{B^2}{2R}g_t^\alpha+Dz_t^\alpha\right]dt+\Sigma dw_t^\alpha+\Sigma_0dW_t^0.
\end{split}\right.
\end{equation}
\end{proposition}

\begin{proof}
For $a.e.\ \alpha\in [0,1]$, we link problem (\ref{Jalphaoptimal}-\ref{xalphaoptimal}) to a stochastic Hamiltonian system as follows
\begin{equation}\label{FBSDE0}
\left\{
  \begin{split}
    & 0=2Ru_t^{\alpha,o}+BP_t^{\alpha},\\
    & dx_t^{\alpha,o}=(Ax_t^{\alpha,o}+Bu_t^{\alpha,o}+Dz_t^\alpha)dt+\Sigma dw_t^\alpha+\Sigma_0dW_t^0, \\
    & dP_t^\alpha=-[AP_t^\alpha+2Q(x_t^{\alpha,o}-\nu_t^\alpha)]dt+q_t^\alpha dw_t^\alpha+q_t^{0\alpha}dW_t^0,\\
    & x_0^{\alpha,o}=x_0^\alpha,\quad P_T^\alpha=2Q_T(x_T^{\alpha,o}-\nu_T^\alpha).
  \end{split}
  \right.
\end{equation}
Now we prove that, if $(x^{\alpha,o},P^\alpha,q^\alpha,q^{0\alpha})\in L_{\mathcal{F}^\alpha}^2(\Omega;C([0,T];\mathbb{R}))\times L_{\mathcal{F}^\alpha}^2(\Omega;C([0,T];\mathbb{R}))\times L_{\mathcal{F}^\alpha}^2(0,T;\mathbb{R})\times L_{\mathcal{F}^\alpha}^2(0,T;\mathbb{R})$ is an adapted solution of (\ref{FBSDE0}), $(x^{\alpha,o},u^{\alpha,o})$ is an optimal pair for problem (\ref{Jalphaoptimal}-\ref{xalphaoptimal}). For any admissible control $u^\alpha$ and the corresponding state $x^\alpha$, we analyze the difference between $J(u^{\alpha,o},z^\alpha)$ and $J(u^\alpha,z^\alpha)$. Since the parameters of cost function is positive (nonnegative), it holds that
\begin{align}\label{costfunctiondifference}
   J(u^{\alpha,o},z^\alpha)-J(u^\alpha,z^\alpha)\leq & E\int_0^T\big[2Q(x_t^{\alpha,o}-x_t^\alpha)(x_t^{\alpha,o}-\nu_t^\alpha)+2Ru_t^{\alpha,o}(u_t^{\alpha,o}-u_t^\alpha)\big]dt\nonumber\\
  &+2Q_TE(x_T^{\alpha,o}-x_T^\alpha)(x_T^{\alpha,o}-\nu_T^\alpha).
\end{align}

Applying It\^{o}'s formula to $P_t^\alpha (x_t^{\alpha,o}-x_t^\alpha)$ on the interval $[0,T]$ and taking expectation, the right-hand side of (\ref{costfunctiondifference}) should be zero. Then we obtain $J(u^{\alpha,o},z^\alpha)-J(u^\alpha,z^\alpha)\leq 0$, which yields that $u^{\alpha,o}$ is an optimal control.

Next, we show that (\ref{FBSDE0}) has an adapted solution. We do this by construction. Let $f_t$ be the solution of the first equation of \eqref{optimalx} (the uniqueness and existence of this Riccati equation can be obtained by \cite[Th.3.5]{Freiling}), and $(g,q^1)$ be a pair given by the assumption in this proposition such that for a.e. $\alpha\in [0,1]$, $(g^\alpha,q^{1\alpha})\in L_{\mathcal{F}^{W^0}}^2(\Omega;C([0,T];\mathbb{R}))\times L_{\mathcal{F}^{W^0}}^2(0,T;\mathbb{R})$ is a solution of (\ref{galpha}). Let $x_t^{\alpha,o}$ be the solution of the SDE
\begin{equation}\label{x+}
	dx_t^{\alpha,o}=\left[\left(A-\frac{B^2}{2R}f_t\right)x_t^{\alpha,o}-\frac{B^2}{2R}g_t^\alpha+Dz_t^\alpha\right]dt+\Sigma dw_t^\alpha+\Sigma_0dW_t^0,\  x_0^{\alpha,o}=x_0^\alpha.
\end{equation}
Denote $P_t^\alpha=f_tx_t^{\alpha,o}+g_t^\alpha$. By It\^{o}'s formula
\begin{align*}
	dP_t^\alpha=&\dot{f}_tx_t^{\alpha,o}dt+f_tdx_t^{\alpha,o}+dg_t^\alpha\\
	=&-[AP_t^\alpha+2Q(x_t^{\alpha,o}-\nu_t^\alpha)]dt+\Sigma f_t dw_t^\alpha+(\Sigma f_t+q_t^{1\alpha})dW_t^0.
\end{align*}
This, combined with \eqref{x+} implies that $(x_t^{\alpha,o},P_t^\alpha,\Sigma f_t,\Sigma f_t+q_t^{1\alpha})$ is an adapted solution of (\ref{FBSDE0}), and $u_t^{\alpha,o}=-\frac{B}{2R}(f_tx_t^{\alpha,o}+g_t^\alpha)$ is an optimal control for the problem (\ref{Jalphaoptimal}-\ref{xalphaoptimal}).

\end{proof}

Define a function $\mu:I\to \mathbb{R}$ by setting
\begin{equation}\label{mu}
\mu(\alpha)=Ex_0^\alpha, \quad \forall \alpha\in I.
\end{equation}

In the sequel, we assume that the function $\mu$ defined by \eqref{mu} is measurable. We note that such a assumption can be satisfied easily. Since $\mu$ is bounded, it follows that $\mu\in L^2[0,1]$.

\begin{proposition}\label{proposition2}
Let $z$ be in $\mathcal{L}_c$, and $(g,q^1)$ be a process satisfies the assumption in Proposition \ref{proposition1}. Define
\begin{equation}\label{conditionalexpectation}
	\varphi_t^\alpha \triangleq \Psi(t)\mu^\alpha+\Psi(t)\int_0^t\Psi(s)^{-1}\Big(-\frac{B^2}{2R}g_s^\alpha+Dz_s^\alpha\Big)ds+\Psi(t)\int_0^t\Psi(s)^{-1}\Sigma_0dW_s^0,
\end{equation}
where $\Psi(t)=\exp\big(\int_0^t (A-\frac{B^2}{2R}f_s)ds\big)$.
Then for a.e. $\alpha$, each $t$, $\varphi_t^\alpha$ is a conditional expectation of $x_t^{\alpha,o}$ given $\mathcal{F}_T^{W^0}$. Moreover, one has for a.e. $\alpha\in[0,1]$, a.s. $\omega\in\Omega$, each $t\in [0,T]$,
$$
z_t^\alpha=\int_0^1M(\alpha,\beta)\varphi_t^\beta d\beta
$$
if and only if for a.e. $\alpha\in[0,1]$, a.s. $\omega\in\Omega$, each $t\in [0,T]$,
\begin{align}\label{0}
	z_t^\alpha= & \Psi(t)\int_0^1 M(\alpha,\beta)\mu^\beta d\beta+\Psi(t)\int_0^t\Psi(s)^{-1} \Big[-\frac{B^2}{2R}\int_0^1 M(\alpha,\beta)g_s^\beta d\beta\Big.\nonumber\\
	& \Big.+D\int_0^1 M(\alpha,\beta)z_s^\beta d\beta\Big]ds+\Psi(t)\int_0^t\Psi(s)^{-1} \Sigma_0 \int_0^1 M(\alpha,\beta)d\beta dW_s^0.
\end{align}
\end{proposition}

\begin{proof}
For $z\in \mathcal{L}_c$ and $(g,q^1)\in\mathcal{L}_c\times\mathcal{L}_L$ satisfies (\ref{galpha}), by (\ref{optimalx}) and applying \cite[Th.3.3]{lvqi}, we have that for a.e. $\alpha\in[0,1]$. a.s. $\omega\in \Omega$,
\begin{align*}
	x_t^{\alpha,o}= & x_0^{\alpha,o}\Psi(t)+\Psi(t)\int_0^t\Psi(s)^{-1} \Big(-\frac{B^2}{2R}g_s^\alpha+Dz_s^\alpha\Big)ds+\Psi(t)\int_0^t\Psi(s)^{-1}\Sigma dw_s^\alpha\\
	& +\Psi(t)\int_0^t\Psi(s)^{-1}\Sigma_0 dW_s^0,\quad \forall t\in[0,T].
\end{align*}
Then, by the definition of $\mathcal{L}_c$ and applying conditional Fubini theorem \cite[Th. 27.17]{Schilling}(or \cite[Lem 2.3]{Brooks}),
it holds that for a.e. $\alpha\in[0,1]$, each $t\in[0,T]$
\begin{align*}
E[x_t^{\alpha,o}|\mathcal{F}_T^{W^0}]= & \Psi(t)\mu^\alpha+\Psi(t)\int_0^t\Psi(s)^{-1}\Big(-\frac{B^2}{2R}g_s^\alpha+Dz_s^\alpha\Big)ds\\
& +\Psi(t)\int_0^t\Psi(s)^{-1}\Sigma_0 dW_s^0,\quad \mathbb{P}.\ a.s.
\end{align*}
which means for a.e. $\alpha\in[0,1]$, each $t\in[0,T]$, $\varphi_t^\alpha$ is a conditional expectation of $x_t^{\alpha,o}$ given $\mathcal{F}_T^{W^0}$.

To  prove the second conclusion, we first note that for a.s. $\omega\in\Omega$, each $t\in[0,T]$, $\varphi_t^\cdot\in L^1 [0,1]$. Indeed, since $z,g\in\mathcal{L}_c$, applying Tonelli theorem and Fubini theorem \cite{Royden}, we have that for a.s. $\omega\in\Omega$, each $t\in[0,T]$, $\int_0^t\Psi(s)^{-1}\big(-\frac{B^2}{2R}g_s^\cdot+Dz_s^\cdot\big)ds$ is in $L^1[0,1]$. From the measurability and boundedness of $\mu^\cdot$, we conclude that for a.s. $\omega\in\Omega$, each $t\in[0,T]$, $\varphi_t^\cdot\in L^1 [0,1]$.

By (\ref{conditionalexpectation}) and using Tonelli theorem and Fubini theorem, it holds that for a.e. $\alpha\in[0,1]$, a.s. $\omega\in\Omega$, each $t\in [0,T]$,
\begin{align}\label{unknown}
	& \int_0^1 M(\alpha,\beta) \varphi_t^\beta d\beta\nonumber\\
	= & \int_0^1 M(\alpha,\beta) \Psi(t)\mu^\beta d\beta+\int_0^1 M(\alpha,\beta) \Psi(t)\int_0^t\Psi(s)^{-1}\Big(-\frac{B^2}{2R}g_s^\beta+Dz_s^\beta\Big)dsd\beta\nonumber\\
	& +\int_0^1 M(\alpha,\beta) \Psi(t)\int_0^t\Psi(s)^{-1}\Sigma_0 dW_s^0 d\beta\nonumber\\
	= & \Psi(t)\int_0^1 M(\alpha,\beta)\mu^\beta d\beta+\Psi(t)\int_0^t\Psi(s)^{-1} \Big[-\frac{B^2}{2R}\int_0^1 M(\alpha,\beta)g_s^\beta d\beta\Big.\nonumber\\
	& \Big.+D\int_0^1 M(\alpha,\beta)z_s^\beta d\beta\Big]ds+\Psi(t)\int_0^t\Psi(s)^{-1} \Sigma_0 \int_0^1 M(\alpha,\beta)d\beta dW_s^0.
\end{align}
If for a.e. $\alpha\in[0,1]$, a.s. $\omega\in\Omega$, each $t\in [0,T]$, $z_t^\alpha=\int_0^1M(\alpha,\beta)\varphi_t^\beta d\beta$, from (\ref{unknown}), we can derive (\ref{0}). For the sufficiency, subtracting (\ref{unknown}) from (\ref{0}), we have that for a.e. $\alpha\in[0,1]$, a.s. $\omega\in\Omega$, each $t\in [0,T]$,
$$
z_t^\alpha-\int_0^1M(\alpha,\beta)\varphi_t^\beta d\beta=0.
$$
This completes the proof.
\end{proof}

Compiling Propositions \ref{proposition1} and \ref{proposition2}, we have that the LQ-GMFG with common noise is solvable, whenever there exists a process $(z,g,q^1)\in \mathcal{L}_c\times\mathcal{L}_c\times\mathcal{L}_L$, which satisfies for a.e. $\alpha\in[0,1]$, each $t\in[0,T]$, a.s. $\omega\in\Omega$,
\begin{equation}\label{FBSDE*}
\left\{
\begin{split}
   z_t^\alpha= & \Psi(t)\int_0^1 M(\alpha,\beta)\mu^\beta d\beta+\Psi(t)\int_0^t\Psi(s)^{-1} \Big[-\frac{B^2}{2R}\int_0^1 M(\alpha,\beta)g_s^\beta d\beta\Big.\\
   & \Big.+D\int_0^1 M(\alpha,\beta)z_s^\beta d\beta\Big]ds+\Psi(t)\int_0^t\Psi(s)^{-1} \Sigma_0 \int_0^1 M(\alpha,\beta)d\beta dW_s^0,\\
    g_t^\alpha=&-2Q_TH(z_T^\alpha+\eta)+\int_t^T\Big[(A-\frac{B^2}{2R}f_s)g_s^\alpha+(Df_s-2QH)z_s^\alpha-2QH\eta\Big]ds\\
    &-\int_t^Tq_s^{1\alpha}dW_s^0.
\end{split}
\right.
\end{equation}

Next, we focus on finding the sufficient conditions ensuring for the existence of the solution for (\ref{FBSDE*}) by applying the spectral decomposition method developed in \cite{Gao4,Gao5}. To this end, we need the following assumption which also appears in \cite{Parise,Gao,Foguen-Tchuendom,Dunyak1,Dunyak2}.

\begin{assumption}\label{finiteeigenvalues}
The graphon operator $M$ corresponding to $M(\cdot,\cdot)$ defined by (\ref{graphonoperator}) admits finite non-zero eigenvalues $\{\lambda_l\}_{l=1}^d$ with a set of orthonormal eigenfunctions $\{f_l\}_{l=1}^d$.
\end{assumption}
The following proposition provides a solution for (\ref{FBSDE*}) with a decomposed form.

\begin{proposition}\label{prop3.3}
For each $l\in\{1,\ldots,d\}$, if there exists a unique triple $(z^l,g^l,q^{1l})\in L_{\mathcal{F}^{W^0}}^2(\Omega;C([0,T];\mathbb{R}))\times L_{\mathcal{F}^{W^0}}^2(\Omega;C([0,T];\mathbb{R}))\times L_{\mathcal{F}^{W^0}}^2(0,T;\mathbb{R})$ satisfying for each $t\in[0,T]$, a.s. $\omega\in\Omega$
\begin{equation}\label{FBSDEl}
\left\{
\begin{split}
   &z_t^l=\lambda_l\langle\mu,f_l\rangle+\int_0^t\left[(A-\frac{B^2}{2R}f_s+D\lambda_l)z_s^l-\frac{B^2}{2R}\lambda_lg_s^l\right]ds+\int_0^t\Sigma_0\lambda_l\langle f_l,\mathds{1}\rangle dW_s^0,\\
   &g_t^l=-2Q_TH(z_T^l+\eta\langle \mathds{1},f_l \rangle)+\int_t^T\left[(A-\frac{B^2}{2R}f_s)g_s^l+(Df_s-2QH)z_s^l-2QH\eta\langle \mathds{1},f_l \rangle\right] ds-\int_t^T q_s^{1l}dW_s^0,
\end{split}
\right.
\end{equation}
then the triple $(z_t^\alpha,g_t^\alpha,q_t^{1\alpha})$ given below is in $\mathcal{L}_c\times\mathcal{L}_c\times\mathcal{L}_L$  and satisfies (\ref{FBSDE*}): for all $\alpha\in [0,1]$, $t\in [0,T]$ and $\omega\in\Omega$,
\begin{equation}\label{couplezgq}
\left\{
\begin{split}
&z_t^\alpha=\sum\limits_{l=1}^d z_t^l f_l(\alpha),\\
&g_t^\alpha=\sum\limits_{l=1}^d g_t^l f_l(\alpha)+\mathring{g}_t\left(\mathds{1}-\sum\limits_{l=1}^d \langle \mathds{1},f_l\rangle f_l\right)(\alpha),\\
&q_t^{1\alpha}=\sum\limits_{l=1}^d q_t^{1l} f_l(\alpha),
\end{split}
\right.
\end{equation}
where $f_l(\cdot)$ is any representation in the equivalence class $f_l$ for each $1\leq l\leq d$, $\mathring{g}$ is the unique solution of the following ODE
\begin{equation}\label{mathringg}
   \frac{d\mathring{g}_t}{dt}=-\left(A-\frac{B^2}{2R}f_t\right)\mathring{g}_t+2QH\eta,\quad \mathring{g}_T=-2Q_TH\eta.
\end{equation}
\end{proposition}
\begin{proof}
It can be easily checked that $(z_t^\alpha,g_t^\alpha,q_t^{1\alpha})$ given by (\ref{couplezgq}) is in $\mathcal{L}_c\times\mathcal{L}_c\times\mathcal{L}_L$. Denote
$$
Y^l(t)\triangleq \lambda_l\langle\mu,f_l\rangle +\int_0^t\Psi(s)^{-1}\Big(D\lambda_l z_s^l-\frac{B^2}{2R}\lambda_l g_s^l\Big)ds+\int_0^t\Psi(s)^{-1}\Sigma_0 \lambda_l \langle f_l, \mathds{1}\rangle dW_s^0.
$$
Using It\^{o}'s formula yields for a.s. $\omega\in\Omega$, each $t\in[0,T]$
\begin{align}\label{unknown1}
	\Psi(t)Y^l(t)= & \lambda_l\langle\mu,f_l\rangle+\int_0^t \Big[(A-\frac{B^2}{2R}f_s)\Psi(s)Y^l(s)+D\lambda_l z_s^l-\frac{B^2}{2R}\lambda_l g_s^l\Big]ds\\
	& +\int_0^t \Sigma_0 \lambda_l \langle f_l, \mathds{1}\rangle dW_s^0.\nonumber
	\end{align}
It can be verified $\Psi(t)Y^l(t)\in L_{\mathcal{F}}^2(\Omega;C([0,T],\mathbb{R}))$. From (\ref{FBSDEl}) and (\ref{unknown1}), one has for a.s. $\omega\in\Omega$, each $t\in[0,T]$
$$
	\Psi(t)Y^l(t)-z_t^l=\int_0^t (A-\frac{B^2}{2R}f_s)(\Psi(s)Y^l(s)-z_s^l)ds
$$
Thus
\begin{equation}\label{unknown2}
\mathbb{P}\big(\Psi(t)Y^l(t)-z_t^l=0,\ t\in[0,T]\big)=1.
\end{equation}
From (\ref{FBSDEl}), (\ref{couplezgq}) and (\ref{unknown2}), since $\int_0^1 M(\alpha,\beta)f_l(\beta)d\beta=\lambda_l f_l(\alpha)$\quad a.e. $\alpha\in[0,1]$, for $1\leq l\leq d$, it holds that for a.e. $\alpha\in[0,1]$, each $t\in[0,T]$, a.s. $\omega\in\Omega$,
\begin{align}\label{z_t^alpha}
	z_t^\alpha= & \sum\limits_{l=1}^d \Psi(t)Y^l(t) f_l(\alpha)\nonumber\\
	= & \sum\limits_{l=1}^d\left\{\lambda_l\langle\mu,f_l\rangle \Psi(t) +\Psi(t)\int_0^t\Psi(s)^{-1}\Big(D\lambda_l z_s^l-\frac{B^2}{2R}\lambda_l g_s^l\Big)ds\right.\nonumber\\
	& \left.+\Psi(t)\int_0^t\Psi(s)^{-1}\Sigma_0 \lambda_l \langle f_l, \mathds{1}\rangle dW_s^0\right\}f_l(\alpha)\nonumber\\
	= & \Psi(t)\sum\limits_{l=1}^d \lambda_l\langle\mu,f_l\rangle f_l(\alpha)+ \Psi(t)\int_0^t\Psi(s)^{-1}\Big(D\sum\limits_{l=1}^d\lambda_l z_s^l f_l(\alpha)-\frac{B^2}{2R}\sum\limits_{l=1}^d\lambda_l g_s^l f_l(\alpha)\Big)ds\nonumber\\
	& +\Psi(t)\int_0^t\Psi(s)^{-1}\Sigma_0 \sum\limits_{l=1}^d \lambda_l \langle f_l, \mathds{1}\rangle f_l(\alpha) dW_s^0\nonumber\\
	= & \Psi(t)\int_0^1 M(\alpha,\beta)\mu^\beta d\beta+\Psi(t)\int_0^t\Psi(s)^{-1} \Big[-\frac{B^2}{2R}\int_0^1 M(\alpha,\beta)g_s^\beta d\beta\Big.\nonumber\\
	& \Big.+D\int_0^1 M(\alpha,\beta)z_s^\beta d\beta\Big]ds+\Psi(t)\int_0^t\Psi(s)^{-1} \Sigma_0 \int_0^1 M(\alpha,\beta)d\beta dW_s^0.
\end{align}

It follows from (\ref{FBSDEl}), (\ref{couplezgq}) and (\ref{mathringg}) that for a.e. $\alpha\in [0,1]$, each $t\in[0,T]$, a.s. $\omega\in\Omega$,
\begin{align}\label{g_t^alpha}
    g_t^\alpha&=\sum\limits_{l=1}^d\Big\{-2Q_TH(z_T^l+\eta\langle \mathds{1},f_l \rangle)+\int_t^T\Big[(A-\frac{B^2}{2R}f_s)g_s^l+(Df_s-2QH)z_s^l\Big.\Big.\nonumber\\
    &\Big.\Big.-2QH\eta\langle \mathds{1},f_l \rangle\Big] ds-\int_t^T q_s^{1l}dW_s^0\Big\}f_l(\alpha)\nonumber\\
    &+\Big\{-2Q_TH\eta+\int_t^T\Big[(A-\frac{B^2}{2R}f_s)\mathring{g}_s-2QH\eta\Big]ds\Big\}(\mathds{1}-\sum\limits_{l=1}^d\langle f_l,\mathds{1}\rangle f_l)(\alpha)\nonumber\\
   &=-2Q_TH(z_T^\alpha+\eta)+\int_t^T\Big[(A-\frac{B^2}{2R}f_s)g_s^\alpha+(Df_s-2QH)z_s^\alpha-2QH\eta\Big]ds\nonumber\\
   &-\int_t^Tq_s^{1\alpha}dW_s^0.
\end{align}
Combining (\ref{z_t^alpha}) and (\ref{g_t^alpha}), the triple given by (\ref{couplezgq}) satisfies (\ref{FBSDE*}).
\end{proof}

Proposition \ref{prop3.3} shows that the LQ-GMFG problem under study has a Nash equilibrium if, for each $l\in\{1,\ldots,d\}$,  (\ref{FBSDEl}) admits a unique solution. To guarantee the unique existence of the solution of (\ref{FBSDEl}), we need the following monotonicity assumption \cite{Xu,Xu1,Peng}.

\begin{assumption}\label{monotonicity}
For each $l\in\{1,\ldots,d\}$, there exist positive constants $\beta_1^l$ and $\mu_1^l$ such that,
\begin{equation*}
\left\{
\begin{split}
&D\lambda_l xy+(2QH-Df_t)x^2-\frac{B^2}{2R}\lambda_l y^2\leq-\beta_1^l x^2,\quad \forall x,y\in \mathbb{R},\\
&Q_TH\leq-\mu_1^l,
\end{split}
\right.
\end{equation*}
or
\begin{equation*}
\left\{
\begin{split}
&D\lambda_l xy+(2QH-Df_t)x^2-\frac{B^2}{2R}\lambda_l y^2\geq\beta_1^l x^2,\quad \forall x,y\in \mathbb{R},\\
&Q_TH\geq\mu_1^l.
\end{split}
\right.
\end{equation*}
\end{assumption}

Under Assumption \ref{monotonicity}, it follows from Theorem 2.6 of Peng and Wu \cite{Peng} that there is a unique solution $(z^l,g^l,q^{1l})\in \left[L_{\mathcal{F}^{W^0}}^2(0,T;\mathbb{R})\right]^3$ satisfying (\ref{FBSDEl}) for each $l\in\{1,\ldots,d\}$. By the standard method (see, for example, \cite{Carmona6}), we can show that there exists a constant $C$ such that for all $l\in\{1,\ldots,d\}$,
\begin{equation}\label{bounded}
  E\sup\limits_{0\leq t\leq T}|z_t^l|^2+E\sup\limits_{0\leq t\leq T}|g_t^l|^2+E\int_0^T|q_t^{1l}|^2dt\leq C.
\end{equation}
Then, we have that $(z^l,g^l,q^{1l})\in L_{\mathcal{F}^{W^0}}^2(\Omega;C([0,T];\mathbb{R}))\times L_{\mathcal{F}^{W^0}}^2(\Omega;C([0,T];\mathbb{R}))\times L_{\mathcal{F}^{W^0}}^2(0,T;\mathbb{R})$. Moreover, under Assumption \ref{monotonicity}, similar to the proof of Theorem 2.2 in \cite{Peng}, if $(\bar{z}^l,\bar{g}^l,\bar{q}^{1l})\in L_{\mathcal{F}^{W^0}}^2(\Omega;C([0,T];\mathbb{R}))\times L_{\mathcal{F}^{W^0}}^2(\Omega;C([0,T];\mathbb{R}))\times L_{\mathcal{F}^{W^0}}^2(0,T;\mathbb{R})$ is another solution of (\ref{FBSDEl}), we derive that
$$
E\sup_t|z_t^l-\bar{z}_t^l|^2+E\sup_t|g_t^l-\bar{g}_t^l|^2+E\int_0^T|q_t^{1l}-\bar{q}_t^{1l}|^2dt=0.
$$
Thus, (\ref{FBSDEl}) has a unique solution $(z^l,g^l,q^{1l})\in L_{\mathcal{F}^{W^0}}^2(\Omega;C([0,T];\mathbb{R}))\times L_{\mathcal{F}^{W^0}}^2(\Omega;C([0,T];\mathbb{R}))\times L_{\mathcal{F}^{W^0}}^2(0,T;\mathbb{R})$.

\begin{remark}
Applying Assumption \ref{monotonicity}, when other parameters are given, we may determine whether a Nash equilibrium exists for the limit GMFG problem solely by observing the eigenvalues of the graphon operator.	
\end{remark}
\begin{remark}
By using Assumption \ref{monotonicity}, we can obtain a Nash equilibrium in an arbitrarily large time duration, which is an improvement compared with some previous work (for example, \cite{Gao,Gao5}).
\end{remark}

Next, we present an alternative approach to establish another sufficient condition for the unique existence of the solution for (\ref{FBSDEl}).

Similar to \cite{Bensoussan1}, we suppose that $g_t^l=K_t^l z_t^l+\Phi_t^l$ for each $l\in\{1,\ldots,d\}$. Then  \eqref{FBSDEl} leads to the following system of ODEs for $K^l$ and $\Phi^l$:
\begin{equation}\label{ODEs}
\left\{
\begin{split}
&\dot{K}_t^l-\frac{B^2}{2R}\lambda_l (K_t^l)^2+(2A-\frac{B^2}{R}f_t+D\lambda_l)K_t^l+Df_t-2QH=0,\quad K_T^l=-2Q_TH,\\
&\dot{\Phi}_t^l+(A-\frac{B^2}{2R}f_t-\frac{B^2}{2R}\lambda_lK_t^l)\Phi_t^l-2QH\eta\langle \mathds{1},f_l\rangle=0,\quad \Phi_T^l=-2Q_TH\eta\langle \mathds{1},f_l\rangle,\\
&q_t^{1l}=\Sigma_0\lambda_l\langle \mathds{1},f_l\rangle K_t^l,
\end{split}
\right.
\end{equation}
which has a unique solution under the following assumption \cite{Gao,Salhab}.
\begin{assumption}\label{assumptionriccati}
For each $l\in\{1,\ldots,d\}$, there exists a solution to the following Riccati equation on $[0,T]$
\begin{equation}\label{riccati}
  -\dot{K}_t^l=-\frac{B^2}{2R}\lambda_l (K_t^l)^2+\left(2A-\frac{B^2}{R}f_t+D\lambda_l\right)K_t^l+Df_t-2QH,\quad K_T^l=-2Q_TH.
\end{equation}
\end{assumption}
\begin{remark}
If Assumption \ref{assumptionriccati} is satisfied, then Riccati equation (\ref{riccati}) on $[0,T]$ has a unique solution due to the smoothness of the righthand side with respect to $K^l$ \cite[Section 2.4, Lemma 1]{Perko}.
\end{remark}
\begin{remark}
The relation between the two sufficient conditions (Assumptions \ref{monotonicity} and \ref{assumptionriccati}) will be studied in future work.
\end{remark}
Therefore, under Assumption \ref{assumptionriccati}, for each $l\in\{1,\ldots,d\}$, (\ref{FBSDEl}) has a unique solution $(z^l,g^l,q^{1l})\in L_{\mathcal{F}^{W^0}}^2(\Omega;C([0,T];\mathbb{R}))\times L_{\mathcal{F}^{W^0}}^2(\Omega;C([0,T];\mathbb{R}))\times L_{\mathcal{F}^{W^0}}^2(0,T;\mathbb{R})$, which also satisfies \eqref{bounded}. Moreover, (\ref{couplezgq}) can be represented by a decoupling form as follows: for all $\alpha\in [0,1]$, $t\in [0,T]$ and $\omega\in\Omega$,
\begin{equation}\label{decouplezgq}
\left\{
\begin{split}
&z_t^\alpha=\sum\limits_{l=1}^d z_t^l f_l(\alpha),\\
&g_t^\alpha=\sum\limits_{l=1}^d (K_t^l z_t^l+\Phi_t^l) f_l(\alpha)+\mathring{g}_t\left(\mathds{1}-\sum\limits_{l=1}^d \langle \mathds{1},f_l\rangle f_l\right)(\alpha),\\
&q_t^{1\alpha}=\sum\limits_{l=1}^d \Sigma_0\lambda_l\langle \mathds{1},f_l\rangle K_t^l f_l(\alpha).
\end{split}
\right.
\end{equation}

Finally, we state the main result in this section as follows.
\begin{theorem}
Under Assumptions \ref{finiteeigenvalues} and \ref{monotonicity} (or \ref{assumptionriccati}), we obtain a solution to the limit graphon mean field game problem as follows: for a.e. $\alpha\in[0,1]$,
\begin{equation}\label{controlalpha1}
  u_t^{\alpha,o}=-\frac{B}{2R}(f_tx_t^{\alpha,o}+g_t^\alpha),
\end{equation}
where $(g_t^\alpha)_{\alpha\in[0,1],t\in[0,T]}$ is given by (\ref{couplezgq}) or (\ref{decouplezgq}), and $f_t$ is given by (\ref{optimalx}).
\end{theorem}

\begin{remark}
By (\ref{couplezgq}), (\ref{bounded}) and (\ref{decouplezgq}),  it is easy to check that the feedback control $u_t^{\alpha,o}$ given by (\ref{controlalpha1}) is indeed admissible for a.e. $\alpha\in [0,1]$.
\end{remark}

\section{$\epsilon$-Nash equilibrium}
In this section, we construct an asymptotic Nash equilibrium (i.e. an $\epsilon$-Nash equilibrium) for Problem \ref{Knash} from the corresponding limit GMFG solutions. To begin with, we give the definition of $\epsilon$-Nash equilibrium as follows.
\begin{definition}\label{epsilonnashdefinition}
For $\epsilon \ge 0$, a set of strategies $(u^{o1},\ldots,u^{oK})$ is called an $\epsilon$-Nash equilibrium with respect to the costs $J_i$, $i\in \{1,\dots,K\}$, if for each $i$, it holds that
$$
J_i(u^{oi},u^{-oi})\leq J_i(v^i,u^{-oi})+\epsilon
$$
for any admissible control $v^i\in\mathcal{U}$.
\end{definition}

We adopt the construction method of $\epsilon$-Nash equilibrium appeared in \cite{Gao}. With the $N$-uniform partition $\{P_1,\ldots,P_N\}$ of $[0,1]$, the $K$-tuple strategies $(u^{o1},\ldots,u^{oK})$ are constructed as follows: for any agent $\mathcal{A}_i$ in $\mathcal{C}_q$,
\begin{equation}\label{epsilonnashstrategy}
\left\{\begin{split}
  &u_t^{oi}=-\frac{B}{2R}(f_tx_t^{oi}+\overline{g}_t^q),\\
  &\overline{g}_t^q\triangleq \frac{1}{\lambda_I(P_q)}\int_{P_q}g_t^\alpha d\alpha=\sum_{l=1}^d g_t^l N\int_{P_q} f_l(\alpha)d\alpha+\mathring{g}_t N\int_{P_q}\Big(\mathds{1}-\sum\limits_{l=1}^d \langle \mathds{1},f_l\rangle f_l\Big)(\alpha)d\alpha,
\end{split}\right.
\end{equation}
where $x^{oi}$ is the corresponding state, $g_t^\alpha$ is given by (\ref{couplezgq}) and (\ref{decouplezgq}), and $\lambda_I(P_q)=1/N$ denotes the Lebesgue measure of $P_q$. The strategy of agent $\mathcal{A}_i$ is decentralized, the agent only need to observe her own state and the common noise.

This section aims to prove that the set of strategies given by (\ref{epsilonnashstrategy}) is an $\epsilon$-Nash equilibrium. To this end, let $\mu^{[N]}$ be the piece-wise constant function in $L^2[0,1]$ with $N$-uniform partition of $[0,1]$ corresponding to $[\mu_1,\ldots,\mu_N]^T$ and $M^{[N]}$ be the step function type graphon corresponding to adjacency matrix $M_N$ of the underlying graph with $N$ nodes.

The following two assumptions will play a key role for deducing our main results.

\begin{assumption}\label{assumptionmu}
The sequence $\{\mu^{[N]}\}$ converges to $\mu$ defined by \eqref{mu} in the norm $\|\cdot\|_1$, i.e.,
$$
\lim\limits_{N\rightarrow\infty}\|\mu^{[N]}-\mu\|_1=0.
$$
\end{assumption}

\begin{assumption}\label{approximation}
The sequence $\{M^{[N]}\}$ and the limit graphon $M$ satisfy
$$
\lim\limits_{N\rightarrow\infty}\max\limits_{q\in\{1\ldots,N\}}\frac{1}{\lambda_I (P_q)}\|(M-M^{[N]})\mathds{1}_{P_q}\|_1=0.
$$
\end{assumption}

We would like to note that the convergence in the cut norm or the cut metric is inadequate for the $\epsilon$-Nash equilibrium analysis in this section. Indeed, as pointed out by Caines and Huang \cite{Caines}, the convergence in the cut norm or the cut metric does not capture sufficiently strong sectional information of the difference between step function type graphon corresponding to the finite graph and the limit graphon. Thus, we adopt a different convergence notion strengthening the sectional requirement as in Assumption \ref{approximation}.

\begin{remark}\label{approximationremark}
It is worth mentioning that Assumptions \ref{assumptionmu} and \ref{approximation} are weaker than the ones in \cite{Gao} due to the fact that $\|\cdot\|_1$ is weaker than $\|\cdot\|_2$.
\end{remark}

The following result indicates that the limit graphon $M$ is well determined under Assumption \ref{approximation}.
\begin{proposition}\label{graphonunique}
For the given sequence $\{M^{[N]}\}$, if there exists a graphon $M$ satisfying Assumption \ref{approximation}, then $M$ is unique.
\end{proposition}
\begin{proof}
We prove the conclusion following the procedure in \cite{Caines} and postpone the details to the appendix.
\end{proof}

We have the following boundedness result for the eigenfunctions of the graphon.
\begin{lemma}\label{boundlemma}
Given a rank $d$ graphon $M(\cdot,\cdot)$, the corresponding eigenfunctions $f_l\in L^\infty [0,1]$, $l\in\{1,\cdots,d\}$, i.e. $\operatorname{ess\,sup}_{\alpha\in[0,1]}|f_l(\alpha)|\leq \frac{1}{\min_{l\in\{1,\cdots,d\}}|\lambda_l|}$ for $l\in\{1,\cdots,d\}$.
\end{lemma}
\begin{proof}
For $l\in\{1,\cdots,d\}$, since $\lambda_l f_l(\alpha)=\int_0^1 M(\alpha,\beta)f_l(\beta)d\beta$, a.e. $\alpha\in[0,1]$, there exists a set $E_l\subset [0,1]$ with $\lambda_I(E_l)=1$ such that for $\alpha\in E_l$, $f_l(\alpha)=\frac{1}{\lambda_l}\int_0^1 M(\alpha,\beta)f_l(\beta)d\beta$. Then, for $\alpha\in E_l$,
\begin{align*}
|f_l(\alpha)| & \leq \frac{1}{|\lambda_l|}\int_0^1 |M(\alpha,\beta)f_l(\beta)|d\beta \leq \frac{1}{|\lambda_l|}\int_0^1 |f_l(\beta)|d\beta\\
& \leq \frac{1}{|\lambda_l|}\Big(\int_0^1 |f_l(\beta)|^2 d\beta\Big)^\frac{1}{2}\leq \frac{1}{\min_{l\in\{1,\cdots,d\}}|\lambda_l|}.
\end{align*}
	
\end{proof}

For any $P_q\subset[0,1]$ with $q=1,\dots,N$, let
$$\overline{z}_t^q\triangleq \frac{1}{\lambda_I(P_q)}\int_{P_q}z_t^\alpha d\alpha,\quad
\overline{q}_t^{1q}\triangleq \frac{1}{\lambda_I(P_q)}\int_{P_q}q_t^{1\alpha} d\alpha,$$
where $z_t^\alpha$ and $q_t^{1\alpha}$ are solution of (\ref{FBSDE*}) and given by (\ref{couplezgq}) and (\ref{decouplezgq}). In what follows, $C$ will denote a generic constant, which may be different in line by line. The following lemma gives the estimation of $\overline{z}_t^q$ and $\overline{q}_t^{1q}$.
\begin{lemma}\label{bounded1} It holds that
$$
\max_{1\leq q\leq N}E\sup_{t\in [0,T]}|\overline{z}_t^q|^2\leq C,\quad \max_{1\leq q\leq N}E\sup_{t\in [0,T]}|\overline{g}_t^q|^2\leq C,
$$
where $C$ is a constant independent of $N$.
\end{lemma}
\begin{proof}
By Lemma \ref{boundlemma}, one has
$$
|\overline{z}_t^q|^2\leq d\sum_{l=1}^d |z_t^l|^2 N\int_{P_q}|f_l(\alpha)|^2d\alpha \leq C\sum_{l=1}^d |z_t^l|^2.
$$
Then, it holds that
$$
E\sup_{0\leq t\leq T}|\overline{z}_t^q|^2\leq C\sum_{l=1}^d E\sup_{0\leq t\leq T}|z_t^l|^2\leq C.
$$
Similarly, we have
$$
E\sup_{0\leq t\leq T}|\overline{g}_t^q|^2\leq C\Big(\sum_{l=1}^d E\sup_{0\leq t\leq T}|g_t^l|^2+\sup_{0\leq t\leq T}|\mathring{g}_t|^2\Big)\leq C.
$$
This complete the proof.
\end{proof}

By (\ref{FBSDEl}), (\ref{couplezgq}), it's easy to show that $(\overline{z}_t^q,\overline{g}_t^q,\overline{q}_t^{1q})$ is the unique solution of the following (decoupled) FBSDE:
\begin{equation}\label{overlinez}
\left\{ \begin{split}
    \overline{z}_t^q=&N\int_{P_q}(M\mu)(\alpha)d\alpha+\int_0^t\Big[(A-\frac{B^2}{2R}f_s)\overline{z}_s^q+N\int_{P_q}D(Mz_s)(\alpha)d\alpha\Big.\\
    &\Big.-\frac{B^2}{2R}N\int_{P_q}(Mg_s)(\alpha)d\alpha\Big]ds+\int_0^tN\int_{P_q}\Sigma_0(M\mathds{1})(\alpha)d\alpha dW_s^0,\\
  \overline{g}_t^q=&-2Q_TH(\overline{z}_T^q+\eta)+\int_t^T\Big[(A-\frac{B^2}{2R}f_s)\overline{g}_s^q-(2QH-Df_s)\overline{z}_s^q-2QH\eta\Big]ds\\
  &-\int_t^T\overline{q}_s^{1q}dW_s^0.
  \end{split}\right.
\end{equation}

In order to show that the set of strategies given by (\ref{epsilonnashstrategy}) is an $\epsilon$-Nash equilibrium, we consider the following auxiliary optimal control problem for agent $\mathcal{A}_i$ in $\mathcal{C}_q$.
\begin{problem}\label{limitcontrol}
To find an optimal control $\overline{u}^i$ of agent $\mathcal{A}_i$ in $\mathcal{C}_q$ such that
\begin{equation*}
  J_i^*(\overline{u}^i)=\inf\limits_{u^i\in\mathcal{U}}E\left\{\int_0^T\left[Q(y_t^i-\overline{\nu}_t^q)^2+R(u_t^i)^2\right]dt+Q_T(y_T^i-\overline{\nu}_T^q)^2\right\},\quad \overline{\nu}_t^q\triangleq H(\overline{z}_t^q+\eta)
\end{equation*}
subject to the following limiting system
\begin{equation}\label{y^i}
\left\{
  \begin{split}
     &dy_t^i=(Ay_t^i+Bu_t^i+D\overline{z}_t^q)dt+\Sigma dw_t^i+\Sigma_0dW_t^0, \\
      &y_0^i=x_0^i.
  \end{split}
  \right.
\end{equation}
\end{problem}

Similar to the proof of Proposition \ref{proposition1}, we have the following result.
\begin{proposition}
There exists an optimal control of the agent $\mathcal{A}_i$ for Problem \ref{limitcontrol} with the form
$\overline{u}_t^i=-\frac{B}{2R}(f_t\overline{y}_t^i+\overline{g}_t^q)$
and the corresponding optimal state $\overline{y}^i$ satisfies
\begin{equation}\label{overliney^i}
\left\{
  \begin{split}
     &d\overline{y}_t^i=\left[\left(A-\frac{B^2}{2R}f_t\right)\overline{y}_t^i-\frac{B^2}{2R}\overline{g}_t^q+D\overline{z}_t^q\right]dt+\Sigma dw_t^i+\Sigma_0dW_t^0, \\
      &\overline{y}_0^i=x_0^i.
  \end{split}
  \right.
\end{equation}
\end{proposition}

Now we continue to prove that the set of strategies given by (\ref{epsilonnashstrategy}) is an $\epsilon$-Nash equilibrium. For given strategy (\ref{epsilonnashstrategy}), it follows from \eqref{1} that the closed-loop system of agent $\mathcal{A}_i$ can be described as follows
\begin{equation}\label{epsilonnashstate}
\left\{
  \begin{split}
     &dx_t^{oi}=\left[\left(A-\frac{B^2}{2R}f_t\right)x_t^{oi}-\frac{B^2}{2R}\overline{g}_t^q+Dz_t^{oi}\right]dt+\Sigma dw_t^i+\Sigma_0dW_t^0, \\
      &x_0^{oi}=x_0^i
  \end{split}
  \right.
\end{equation}
where
$$
z_t^{oi}\triangleq\frac{1}{N}\sum\limits_{l=1}^N m_{ql}\frac{1}{|\mathcal{C}_l|}\sum\limits_{j\in\mathcal{C}_l}x_t^{oj}.
$$

Since $z_t^{oi}=z_t^{oj}$ for any $i,j\in \mathcal{C}_q$, denoting $z_t^{oq}=z_t^{oi}$ for all $i\in \mathcal{C}_q$,  we have the following estimation results for $z_t^{oq}$ and $x_t^{oi}$.
\begin{lemma}\label{bounded2} It holds that
$$
\max\limits_{1\leq q\leq N}\sup\limits_{0\leq t\leq T}E|z_t^{oq}|^2\leq C,\quad \max\limits_{1\leq i\leq K}\sup\limits_{0\leq t\leq T}E|x_t^{oi}|^2\leq C,
$$
where $C$ is a constant independent of $N$ and $K$.
\end{lemma}
\begin{proof}
For any $q$ with $1\leq q\leq N$, one has
$$
|z_t^{oq}|^2=\left|\frac{1}{N}\sum\limits_{l=1}^N m_{ql}\frac{1}{|\mathcal{C}_l|}\sum\limits_{j\in\mathcal{C}_l}x_t^{oj}\right|^2\leq \frac{1}{N}\sum\limits_{l=1}^N \frac{1}{|\mathcal{C}_l|}\sum\limits_{j\in\mathcal{C}_l}|x_t^{oj}|^2.
$$
Taking square on both sides of (\ref{epsilonnashstate}) leads to
\begin{equation}\label{square}
  |x_t^{oi}|^2\leq C_1\left[|x_0^i|^2+\int_0^t\left(|x_s^{oi}|^2+|\overline{g}_s^q|^2+|z_s^{oq}|^2\right)ds+\left|\int_0^t\Sigma dw_s^i\right|^2+\left|\int_0^t\Sigma_0 dW_s^0\right|^2\right]
\end{equation}
and so
\begin{equation*}
  \begin{split}
    \frac{1}{N}\sum\limits_{l=1}^N \frac{1}{|\mathcal{C}_l|}\sum\limits_{j\in\mathcal{C}_l}|x_t^{oj}|^2 & \leq C_1\left[\frac{1}{N}\sum\limits_{l=1}^N \frac{1}{|\mathcal{C}_l|}\sum\limits_{j\in\mathcal{C}_l}|x_0^j|^2
+\int_0^t \left(\frac{1}{N}\sum\limits_{l=1}^N \frac{1}{|\mathcal{C}_l|}\sum\limits_{j\in\mathcal{C}_l}|x_s^{oj}|^2
+\frac{1}{N}\sum\limits_{l=1}^N |\overline{g}_s^l|^2\right) ds\right.\\
&\left.+\frac{1}{N}\sum\limits_{l=1}^N \frac{1}{|\mathcal{C}_l|}\sum\limits_{j\in\mathcal{C}_l}\left|\int_0^t\Sigma dw_s^j\right|^2
+\left|\int_0^t\Sigma_0 dW_s^0\right|^2\right].
  \end{split}
\end{equation*}
By Gronwall's inequality, we obtain
\begin{equation*}
  \begin{split}
    E\left(\frac{1}{N}\sum\limits_{l=1}^N \frac{1}{|\mathcal{C}_l|}\sum\limits_{j\in\mathcal{C}_l}|x_t^{oj}|^2\right)
    &\leq C_1\left[\frac{1}{N}\sum\limits_{l=1}^N \frac{1}{|\mathcal{C}_l|}\sum\limits_{j\in\mathcal{C}_l}E|x_0^j|^2
    +E\int_0^T\frac{1}{N}\sum\limits_{l=1}^N |\overline{g}_s^l|^2 ds\right.\\
    &\left.+\frac{1}{N}\sum\limits_{l=1}^N \frac{1}{|\mathcal{C}_l|}\sum\limits_{j\in\mathcal{C}_l}E\sup\limits_{0\leq t\leq T}\left|\int_0^t\Sigma dw_s^j\right|^2
+E\sup\limits_{0\leq t\leq T}\left|\int_0^t\Sigma_0 dW_s^0\right|^2\right].
  \end{split}
\end{equation*}
It follows from the BDG inequality \cite{Revuz} and Lemma \ref{bounded1} that
$$
E|z_t^{oq}|^2\leq E\left(\frac{1}{N}\sum\limits_{l=1}^N \frac{1}{|\mathcal{C}_l|}\sum\limits_{j\in\mathcal{C}_l}|x_t^{oj}|^2\right)\leq C,
$$
where $C$ is independent of $N$ and $K$. Then we have the first conclusion. Taking expectation on both sides of (\ref{square}) and applying the Gronwall's inequality again, we derive that
\begin{equation*}
  \begin{split}
     E|x_t^{oi}|^2\leq& C_1\left[E|x_0^i|^2+\int_0^tE|x_s^{oi}|^2ds +E\int_0^T\left(|\overline{g}_s^q|^2+|z_s^{oq}|^2\right)ds\right.\\
     &\left.+E\sup\limits_{0\leq t\leq T}\left|\int_0^t\Sigma dw_s^i\right|^2+E\sup\limits_{0\leq t\leq T}\left|\int_0^t\Sigma_0 dW_s^0\right|^2\right]\\
     \leq& C_1\left[E|x_0^i|^2+E\int_0^T \left(|\overline{g}_s^q|^2+|z_s^{oq}|^2+\Sigma^2+\Sigma_0^2\right)ds\right].
   \end{split}
\end{equation*}
By the boundedness of $\overline{g}^q$ and $z^{oq}$, we can deduce the second conclusion.
\end{proof}

Next we show the following lemma.

\begin{lemma}\label{norminfty}
If $A_N=[a_{ij}]\in \mathbb{R}^{N\times N}$ satisfies $|a_{ij}|\leq C_a$ for all $i, j\in\{1,\ldots,N\}$, then $\left\|e^{\frac{\gamma}{N}A_N}\right\|_\infty\leq e^{|\gamma|C_a}$ for all $\gamma\in \mathbb{R}$.
\end{lemma}
\begin{proof}
For the case $i=j$, one has
\begin{equation*}
  \begin{split}
    \left|\left[e^{\frac{\gamma}{N}A_N}\right]_{ii}\right|&=\left|\left[I_N+\frac{\gamma}{N}A_N+\frac{1}{2!}\frac{\gamma^2}{N^2}A_N^2+\cdots\right]_{ii}\right|\\
    &\leq 1+\frac{|\gamma|}{N}|[A_N]_{ii}|+\frac{1}{2!}\frac{|\gamma|^2}{N^2}|[A_N^2]_{ii}|+\cdots\\
    &\leq 1+\frac{|\gamma|}{N}C_a+\frac{1}{2!}\frac{|\gamma|^2}{N}C_a^2+\frac{1}{3!}\frac{|\gamma|^3}{N}C_a^3+\cdots\\
    &\leq \frac{1}{N}\left(1+|\gamma|C_a+\frac{1}{2!}|\gamma|^2C_a^2+\frac{1}{3!}|\gamma|^3C_a^3+\cdots\right)+1-\frac{1}{N}\\
    &=\frac{1}{N}e^{|\gamma|C_a}+1-\frac{1}{N}.
  \end{split}
\end{equation*}
For the case $i\neq j$, we have
\begin{equation*}
  \begin{split}
   \left|\left[e^{\frac{\gamma}{N}A_N}\right]_{ij}\right|&=\left|\left[I_N+\frac{\gamma}{N}A_N+\frac{1}{2!}\frac{\gamma^2}{N^2}A_N^2+\cdots\right]_{ij}\right|\\
   &\leq \frac{|\gamma|}{N}|[A_N]_{ij}|+\frac{1}{2!}\frac{|\gamma|^2}{N^2}|[A_N^2]_{ij}|+\cdots\\
   &\leq \frac{|\gamma|}{N}C_a+\frac{1}{2!}\frac{|\gamma|^2}{N}C_a^2+\frac{1}{3!}\frac{|\gamma|^3}{N}C_a^3+\cdots\\
   &=\frac{1}{N}e^{|\gamma|C_a}-\frac{1}{N}.
  \end{split}
\end{equation*}
Therefore
$$
\left\|e^{\frac{\gamma}{N}A_N}\right\|_\infty=\max\limits_i\sum\limits_{j=1}^N\left|\left[e^{\frac{\gamma}{N}A_N}\right]_{ij}\right|\leq e^{|\gamma|C_a}.
$$
\end{proof}
For convenience, denote $$E_N\triangleq\max\limits_{q\in\{1\ldots,N\}}\frac{1}{\lambda_I (P_q)}\|(M-M^{[N]})\mathds{1}_{P_q}\|_1, \quad E_N'\triangleq\|\mu^{[N]}-\mu\|_1, \quad \delta_K\triangleq {E_N}^2+(E_N')^2+\frac{1}{\min_{1\leq l\leq N}|\mathcal{C}_l|}.$$ By Assumptions \ref{assumptionmu} and \ref{approximation}, one has $E_N,E_N'\rightarrow 0$, as $N\rightarrow \infty$, which means $\delta_K\rightarrow 0$, as $N\rightarrow\infty$, $\min_l|\mathcal{C}_l|\rightarrow\infty$. We say that a function $f(N,|\mathcal{C}_1|,\ldots,|\mathcal{C}_N|)$ is $O(\delta_K)$ (or $O({\delta_K}^\frac{1}{2})$) if there are constants $C, N_0, k_0$, such that
$$
|f(N,|\mathcal{C}_1|,\ldots,|\mathcal{C}_N|)|\leq C\delta_K \ (\mbox{ or }|f(N,|\mathcal{C}_1|,\ldots,|\mathcal{C}_N|)|\leq C{\delta_K}^\frac{1}{2}),
$$
for all $N\geq N_0$, $\min_l|\mathcal{C}_l|\geq k_0$. Now we present the approximation between the limiting system and the closed-loop system.

\begin{proposition}\label{error1}
For $q\in\{1,\ldots,N\}$, $i\in\{1,\ldots,K\}$, it holds that
\begin{enumerate}[(i)]
\item $\sup\limits_{0\leq t\leq T}E\left|z_t^{oq}-\overline{z}_t^q\right|^2=O(\delta_K)$;
\item $\sup\limits_{0\leq t\leq T}E \left||z_t^{oq}|^2-|\overline{z}_t^q|^2\right|=O({\delta_K}^{\frac{1}{2}})$;
\item $\sup\limits_{0\leq t\leq T}E\left|x_t^{oi}-\overline{y}_t^i\right|^2=O(\delta_K)$;
\item $\sup\limits_{0\leq t\leq T}E \left||x_t^{oi}|^2-|\overline{y}_t^i|^2\right|=O({\delta_K}^{\frac{1}{2}})$.
\end{enumerate}
\end{proposition}
\begin{proof}
For all $\mathcal{A}_j\in\mathcal{C}_l$, one has
\begin{equation}\label{xoj}
  dx_t^{oj}=\left[(A-\frac{B^2}{2R}f_t)x_t^{oj}-\frac{B^2}{2R}\overline{g}_t^l+D\frac{1}{N}\sum\limits_{k=1}^N m_{lk}\frac{1}{|\mathcal{C}_k|}\sum\limits_{n\in\mathcal{C}_k}x_t^{on}\right]dt+\Sigma dw_t^j+\Sigma_0dW_t^0.
\end{equation}
In addition, for $l\in\{1,\ldots,N\}$, let $\widehat{x}_t^l$ be the unique solution of the following equation
\begin{equation}\label{widetildexl}
  \widehat{x}_t^l=\mu_l+\int_0^t \left[(A-\frac{B^2}{2R}f_s)\widehat{x}_s^l-\frac{B^2}{2R}\overline{g}_s^l+D\frac{1}{N}\sum\limits_{k=1}^N m_{lk}\widehat{x}_s^k\right]ds+\int_0^t\Sigma_0dW_s^0.
\end{equation}
Let $\widehat{z}_t^q\triangleq\frac{1}{N}\sum_{l=1}^N m_{ql}\widehat{x}_t^l$. Then
\begin{equation}\label{widehat{z}_t^q}
  \widehat{z}_t^q=\frac{1}{N}\sum\limits_{l=1}^N m_{ql}\mu_l+\int_0^t \left[(A-\frac{B^2}{2R}f_s)\widehat{z}_s^q-\frac{B^2}{2R}\frac{1}{N}\sum\limits_{l=1}^N m_{ql}\overline{g}_s^l+D\frac{1}{N}\sum\limits_{l=1}^N m_{ql}\widehat{z}_s^l\right]ds+\int_0^t\frac{1}{N}\sum\limits_{l=1}^N m_{ql}\Sigma_0dW_s^0.
\end{equation}
By the triangle inequality, one has
\begin{equation}\label{triangle}
  E|\overline{z}_t^q-z_t^{oq}|^2\leq 2E|\overline{z}_t^q-\widehat{z}_t^q|^2+2E|\widehat{z}_t^q-z_t^{oq}|^2.
\end{equation}
Inspired by the proof of (105)-(108) in Gao et al. \cite{Gao}, we first deal with the first part of the right-hand side of inequality (\ref{triangle}). Let $\Delta_t^q\triangleq\overline{z}_t^q-\widehat{z}_t^q$ and recall $\mu^{[N]}=\sum_{l=1}^N\mu_l\mathds{1}_{P_l}$. Then it follows from \eqref{overlinez} and \eqref{widehat{z}_t^q} that
\begin{align}\label{inequality}
   |\Delta_0^q|&=|\overline{z}_0^q-\widehat{z}_0^q|\nonumber\\
   &=\left|N\int_{P_q}(M\mu)(\alpha)d\alpha-\frac{1}{N}\sum\limits_{l=1}^N m_{ql}\mu_l\right|\nonumber\\
  &=\left|N\int_{P_q}(M\mu-M^{[N]}\mu^{[N]})(\alpha)d\alpha\right|\nonumber\\
  &\leq\left|N\int_{P_q}\left((M-M^{[N]})\mu\right)(\alpha)d\alpha\right|+\left|N\int_{P_q}\left(M^{[N]}(\mu-\mu^{[N]})\right)(\alpha)d\alpha\right|.
\end{align}
By the boundedness of $\mu(\cdot)$, we have
\begin{align*}
 \left|N\int_{P_q}\left((M-M^{[N]})\mu\right)(\alpha)d\alpha\right|&=N\left|\int_{P_q}\int_0^1(M-M^{[N]})(\alpha,\beta)\mu(\beta)d\beta d\alpha\right|\\
&=N\left|\int_0^1\mu(\beta)\int_{P_q}(M-M^{[N]})(\alpha,\beta)d\alpha d\beta\right|\\
&\leq N\int_0^1\left|\mu(\beta)\right|\cdot\left|\int_{P_q}(M-M^{[N]})(\alpha,\beta)d\alpha\right| d\beta\\
&\leq C\cdot N\int_0^1\left|\int_{P_q}(M-M^{[N]})(\alpha,\beta)d\alpha\right| d\beta\\
&\leq CE_N.
\end{align*}
For the second part of the right side of inequality (\ref{inequality}), one has
\begin{align*}
  \left|N\int_{P_q}\left(M^{[N]}(\mu-\mu^{[N]})\right)(\alpha)d\alpha\right|&=\left|N\int_{P_q}\int_0^1\sum\limits_{l=1}^N \mathds{1}_{P_l}(\beta)m_{ql}(\mu-\mu^{[N]})(\beta)d\beta d\alpha\right|\\
&=\left|\int_0^1\sum\limits_{l=1}^N \mathds{1}_{P_l}(\beta)m_{ql}(\mu-\mu^{[N]})(\beta)d\beta\right|\\
&=\left|\sum\limits_{l=1}^Nm_{ql}\int_{P_l}(\mu-\mu^{[N]})(\beta)d\beta\right|\\
&\leq\int_0^1\left|(\mu-\mu^{[N]})(\beta)\right|d\beta\\
&=E_N'.
\end{align*}
Thus, above two inequalities lead to
\begin{equation}\label{initial}
  |\Delta_0^q|\leq C(E_N+E_N').
\end{equation}
Now consider $\Delta_t^q$. Let $\widehat{z}_t\triangleq\sum_{l=1}^N \widehat{z}_t^l \mathds{1}_{P_l}$. Similarly we define $\overline{z}_t$ and $\overline{g}_t$. Then it yields
\begin{align}\label{Deltaq}
    \Delta_t^q=&\Delta_0^q+\int_0^t\left[(A-\frac{B^2}{2R}f_s)\Delta_s^q+DN\int_{P_q}(Mz_s-M^{[N]}\widehat{z}_s)(\alpha)d\alpha-\frac{B^2}{2R}N\int_{P_q}(Mg_s-M^{[N]}\overline{g}_s)(\alpha)d\alpha\right]ds\nonumber\\
&\mbox{}+\int_0^t\left[N\int_{P_q}\Sigma_0(M\mathds{1})(\alpha)d\alpha-\frac{1}{N}\sum\limits_{l=1}^Nm_{ql}\Sigma_0\right]dW_s^0\nonumber\\
=&\Delta_0^q+\int_0^t\left[(A-\frac{B^2}{2R}f_s)\Delta_s^q+DN\int_{P_q}((M-M^{[N]})z_s)(\alpha)d\alpha-\frac{B^2}{2R}N\int_{P_q}((M-M^{[N]})g_s)(\alpha)d\alpha\right.\nonumber\\
&\left.\mbox{}+D\frac{1}{N}\sum\limits_{l=1}^Nm_{ql}\Delta_s^l\right]ds+\int_0^t\Sigma_0N\int_{P_q}((M-M^{[N]})\mathds{1})(\alpha)d\alpha dW_s^0.
\end{align}
Denoting $\widetilde{\Delta}_t^N\triangleq [\Delta_t^1,\ldots,\Delta_t^N]^T$, we have
\begin{equation}\label{Delta}
  \widetilde{\Delta}_t^N=\widetilde{\Delta}_0^N+\int_0^t\left\{\left[(A-\frac{B^2}{2R}f_s)I_N+\frac{D}{N}M_N\right]\widetilde{\Delta}_s^N+D\cdot D_s^{Nz}-\frac{B^2}{2R}\cdot D_s^{Ng}\right\}ds+\int_0^t\Sigma_0D^{N\mathds{1}}dW_s^0
\end{equation}
where $I_N$ is the identity matrix in $\mathbb{R}^{N\times N}$ and
\begin{align*}
  &D_s^{Nz}=N\left\langle \mathds{1}_{P_q},(M-M^{[N]})z_s\right\rangle_{q=1}^N, \\
  &D_s^{Ng}=N\left\langle \mathds{1}_{P_q},(M-M^{[N]})g_s\right\rangle_{q=1}^N, \\
  &D^{N\mathds{1}}=N\left\langle \mathds{1}_{P_q},(M-M^{[N]})\mathds{1}\right\rangle_{q=1}^N.
\end{align*}
By applying Theorem 3.3 in \cite{lvqi}, the solution $\widetilde{\Delta}_t^N$ of (\ref{Delta}) is given as follows
\begin{equation}\label{Deltare}
  \widetilde{\Delta}_t^N=\Phi(t)\widetilde{\Delta}_0^N+\Phi(t)\int_0^t\Phi(s)^{-1}(D\cdot D_s^{Nz}-\frac{B^2}{2R}\cdot D_s^{Ng})ds+\Phi(t)\int_0^t\Phi(s)^{-1}\Sigma_0D^{N\mathds{1}}dW_s^0,
\end{equation}
where
\begin{align*}
  \Phi(t)=\Gamma(t)\exp\left(\frac{Dt}{N}M_N\right),\quad \Gamma(t)=\exp\left(\int_0^t(A-\frac{B^2}{2R}f_s)ds\right).
\end{align*}
From (\ref{Deltare}), one has
\begin{align}\label{Deltasquare}
  &\widetilde{\Delta}_t^N(\widetilde{\Delta}_t^N)^T\nonumber\\
  =&\Phi(t)\widetilde{\Delta}_0^N(\widetilde{\Delta}_0^N)^T{\Phi(t)}^T+\Phi(t)\widetilde{\Delta}_0^N\left(\Phi(t)\int_0^t\Phi(s)^{-1}(D\cdot D_s^{Nz}-\frac{B^2}{2R}\cdot D_s^{Ng})ds\right)^T\nonumber\\
  &+\Phi(t)\widetilde{\Delta}_0^N\left(\Phi(t)\int_0^t\Phi(s)^{-1}\Sigma_0D^{N\mathds{1}}dW_s^0\right)^T+\Phi(t)\int_0^t\Phi(s)^{-1}(D\cdot D_s^{Nz}-\frac{B^2}{2R}\cdot D_s^{Ng})ds(\widetilde{\Delta}_0^N)^T{\Phi(t)}^T\nonumber\\
  &+\Phi(t)\int_0^t\Phi(s)^{-1}(D\cdot D_s^{Nz}-\frac{B^2}{2R}\cdot D_s^{Ng})ds\left(\Phi(t)\int_0^t\Phi(s)^{-1}(D\cdot D_s^{Nz}-\frac{B^2}{2R}\cdot D_s^{Ng})ds\right)^T\nonumber\\
  &+\Phi(t)\int_0^t\Phi(s)^{-1}(D\cdot D_s^{Nz}-\frac{B^2}{2R}\cdot D_s^{Ng})ds\left(\Phi(t)\int_0^t\Phi(s)^{-1}\Sigma_0D^{N\mathds{1}}dW_s^0\right)^T\nonumber\\
  &+\Phi(t)\int_0^t\Phi(s)^{-1}\Sigma_0D^{N\mathds{1}}dW_s^0(\widetilde{\Delta}_0^N)^T{\Phi(t)}^T\nonumber\\
  &+\Phi(t)\int_0^t\Phi(s)^{-1}\Sigma_0D^{N\mathds{1}}dW_s^0\left(\Phi(t)\int_0^t\Phi(s)^{-1}(D\cdot D_s^{Nz}-\frac{B^2}{2R}\cdot D_s^{Ng})ds\right)^T\nonumber\\
  &+\Phi(t)\int_0^t\Phi(s)^{-1}\Sigma_0D^{N\mathds{1}}dW_s^0\left(\Phi(t)\int_0^t\Phi(s)^{-1}\Sigma_0D^{N\mathds{1}}dW_s^0\right)^T\nonumber\\
  \triangleq&\sum\limits_{i=1}^9Y_i.
\end{align}
By (\ref{initial}) and Lemma 4 in \cite{Gao}, $Y_1$ satisfies
\begin{align}\label{EY_1}
    E[Y_1]_{qq}=[\Phi(t)\widetilde{\Delta}_0^N(\widetilde{\Delta}_0^N)^T{\Phi(t)}^T]_{qq}=\left|[\Phi(t)\widetilde{\Delta}_0^N]_q\right|^2
\leq C{\Gamma(t)}^2e^{2|D|t}(E_N+E_N')^2\leq C(E_N+E_N')^2.
\end{align}
For $Y_3$, denoting the $q$th row of $\Phi(t)$ by $\Phi(t)_{q:}$, one has
\begin{equation*}
  \begin{split}
    [Y_3]_{qq}&=\left[\Phi(t)\widetilde{\Delta}_0^N\left(\Phi(t)\int_0^t\Phi(s)^{-1}\Sigma_0D^{N\mathds{1}}dW_s^0\right)^T\right]_{qq}\\
    &=[\Phi(t)\widetilde{\Delta}_0^N]_q\left[\Phi(t)\int_0^t\Phi(s)^{-1}\Sigma_0D^{N\mathds{1}}dW_s^0\right]_q\\
    &=[\Phi(t)\widetilde{\Delta}_0^N]_q \Phi(t)_{q:}\int_0^t\Phi(s)^{-1}\Sigma_0D^{N\mathds{1}}dW_s^0
  \end{split}
\end{equation*}
and so
\begin{equation}\label{EY_3}
  E[Y_3]_{qq}=[\Phi(t)\widetilde{\Delta}_0^N]_q E\left[\Phi(t)_{q:}\int_0^t\Phi(s)^{-1}\Sigma_0D^{N\mathds{1}}dW_s^0\right]=0.
\end{equation}
Clearly, $E[Y_7]_{qq}=E[Y_3]_{qq}=0$. Moreover, by the boundedness of eigenfunctions $f_l$, for $l\in\{1,\ldots,d\}$, we have
\begin{align}\label{DNz}
    \|D_s^{Nz}\|_\infty&=\max\limits_{1\leq q\leq N}\left|N\int_{P_q}((M-M^{[N]})z_s)(\alpha)d\alpha\right|\nonumber\\
&=\max\limits_{1\leq q\leq N}\left|N\int_{P_q}\int_0^1(M-M^{[N]})(\alpha,\beta)z_s(\beta)d\beta d\alpha\right|\nonumber\\
&\leq \max\limits_{1\leq q\leq N}N\int_0^1|z_s(\beta)|\cdot|\int_{P_q}(M-M^{[N]})(\alpha,\beta)d\alpha|d\beta\nonumber\\
&\leq \max\limits_{1\leq q\leq N}N\int_0^1\left(\sum\limits_{l=1}^d|z_s^l|\cdot|f_l(\beta)|\right)\cdot|\int_{P_q}(M-M^{[N]})(\alpha,\beta)d\alpha|d\beta\nonumber\\
&\leq C\left(\sum\limits_{l=1}^d|z_s^l|\right)\max\limits_{1\leq q\leq N}N\int_0^1|\int_{P_q}(M-M^{[N]})(\alpha,\beta)d\alpha|d\beta\nonumber\\
&=C\left(\sum\limits_{l=1}^d|z_s^l|\right)E_N.
\end{align}
Then
\begin{equation*}
  E\int_0^t\|D_s^{Nz}\|_\infty^2 ds\leq C{E_N}^2.
\end{equation*}
Similarly, we obtain $E\int_0^t\|D_s^{Ng}\|_\infty^2 ds\leq C{E_N}^2$. For the fifth part of (\ref{Deltasquare}),
\begin{align}\label{Y_5}
   [Y_5]_{qq}&=\left[\Phi(t)\int_0^t\Phi(s)^{-1}(D\cdot D_s^{Nz}-\frac{B^2}{2R}\cdot D_s^{Ng})ds\left(\Phi(t)\int_0^t\Phi(s)^{-1}(D\cdot D_s^{Nz}-\frac{B^2}{2R}\cdot D_s^{Ng})ds\right)^T\right]_{qq}\nonumber\\
&=\left[\Phi(t)\int_0^t\Phi(s)^{-1}(D\cdot D_s^{Nz}-\frac{B^2}{2R}\cdot D_s^{Ng})ds\right]_q^2\nonumber\\
&=\left(\int_0^t[\Phi(t)\Phi(s)^{-1}]_{q:}(D\cdot D_s^{Nz}-\frac{B^2}{2R}\cdot D_s^{Ng})ds\right)^2\nonumber\\
&\leq\left(\int_0^t\sum\limits_{l=1}^N\left|[\Phi(t)\Phi(s)^{-1}]_{ql}\right|\cdot\left|[D\cdot D_s^{Nz}-\frac{B^2}{2R}\cdot D_s^{Ng}]_l\right|ds\right)^2\nonumber\\
&\leq\left(\int_0^t\sum\limits_{l=1}^N\left|[\Phi(t)\Phi(s)^{-1}]_{ql}\right|\cdot\|D\cdot D_s^{Nz}-\frac{B^2}{2R}\cdot D_s^{Ng}\|_{\infty}ds\right)^2\nonumber\\
&\leq\left(\int_0^t\|\Phi(t)\Phi(s)^{-1}\|_{\infty}\cdot\|D\cdot D_s^{Nz}-\frac{B^2}{2R}\cdot D_s^{Ng}\|_{\infty}ds\right)^2.
\end{align}
By Lemma \ref{norminfty}, we derive
\begin{equation}\label{Phiinfty}
  \|\Phi(t)\Phi(s)^{-1}\|_{\infty}=\Gamma(t){\Gamma(s)}^{-1}\|e^{\frac{D(t-s)}{N}M_N}\|_\infty\leq \Gamma(t){\Gamma(s)}^{-1}e^{|D|(t-s)}.
\end{equation}
By (\ref{DNz}), (\ref{Y_5}) and (\ref{Phiinfty}), we have
\begin{align}\label{EY_5}
    E[Y_5]_{qq}&\leq E\left(\int_0^t\|\Phi(t)\Phi(s)^{-1}\|_{\infty}\cdot\|D\cdot D_s^{Nz}-\frac{B^2}{2R}\cdot D_s^{Ng}\|_{\infty}ds\right)^2\nonumber\\
    &\leq 2TE\int_0^t{\Gamma(t)}^2{\Gamma(s)}^{-2}e^{2|D|(t-s)}\left(D^2\|D_s^{Nz}\|_\infty^2+\frac{B^4}{4R^2}\|D_s^{Ng}\|_\infty^2\right)ds\nonumber\\
    &\leq CE\int_0^t(\|D_s^{Nz}\|_\infty^2+\|D_s^{Ng}\|_\infty^2)ds\nonumber\\
    &\leq C{E_N}^2.
\end{align}
Hence, for the second part of (\ref{Deltasquare}),
\begin{equation*}
  \begin{split}
    [Y_2]_{qq}&=\left[\Phi(t)\widetilde{\Delta}_0^N\left(\Phi(t)\int_0^t\Phi(s)^{-1}(D\cdot D_s^{Nz}-\frac{B^2}{2R}\cdot D_s^{Ng})ds\right)^T\right]_{qq}\\
    &=\left[\Phi(t)\widetilde{\Delta}_0^N\right]_q\left[\Phi(t)\int_0^t\Phi(s)^{-1}(D\cdot D_s^{Nz}-\frac{B^2}{2R}\cdot D_s^{Ng})ds\right]_q\\
    &\leq\frac{1}{2}\left|\left[\Phi(t)\widetilde{\Delta}_0^N\right]_q\right|^2+\frac{1}{2}\left|\left[\Phi(t)\int_0^t\Phi(s)^{-1}(D\cdot D_s^{Nz}-\frac{B^2}{2R}\cdot D_s^{Ng})ds\right]_q\right|^2\\
    &=\frac{1}{2}[Y_1]_{qq}+\frac{1}{2}[Y_5]_{qq}.
  \end{split}
\end{equation*}
Then
\begin{equation}\label{EY_2}
  E[Y_2]_{qq}\leq\frac{1}{2}E[Y_1]_{qq}+\frac{1}{2}E[Y_5]_{qq}.
\end{equation}
This same bound holds for the fourth part $Y_4$ of (\ref{Deltasquare}), i.e. $E[Y_4]_{qq}=E[Y_2]_{qq}.$ For the ninth part of (\ref{Deltasquare}), denoting the $q$th row of $\Phi(t)$ by $\Phi(t)_{q:}$, one has
\begin{align}\label{Y_9}
   [Y_9]_{qq}&=\left[\Phi(t)\int_0^t\Phi(s)^{-1}\Sigma_0D^{N\mathds{1}}dW_s^0\left(\Phi(t)\int_0^t\Phi(s)^{-1}\Sigma_0D^{N\mathds{1}}dW_s^0\right)^T\right]_{qq}\nonumber\\
&=\left[\int_0^t\Phi(t)\Phi(s)^{-1}\Sigma_0D^{N\mathds{1}}dW_s^0\right]_q^2\nonumber\\
&=\Sigma_0^2\left(\int_0^t[\Phi(t)\Phi(s)^{-1}D^{N\mathds{1}}]_q dW_s^0\right)^2\nonumber\\
&=\Sigma_0^2\left(\int_0^t\Phi(t)_{q:}(\Phi(s)^{-1}D^{N\mathds{1}}) dW_s^0\right)^2\nonumber\\
&=\Sigma_0^2\left(\sum\limits_{l=1}^N[\Phi(t)]_{ql}\int_0^t[\Phi(s)^{-1}D^{N\mathds{1}}]_l dW_s^0\right)^2.
\end{align}
We note $\|D^{N\mathds{1}}\|_\infty\leq E_N$. Then, by Lemma \ref{norminfty}, one has
\begin{align}\label{EY_9}
   E[Y_9]_{qq}&=\Sigma_0^2E\left(\sum\limits_{l=1}^N[\Phi(t)]_{ql}\int_0^t[\Phi(s)^{-1}D^{N\mathds{1}}]_l dW_s^0\right)^2\nonumber\\
   &=\Sigma_0^2\sum\limits_{l=1}^N \sum\limits_{k=1}^N [\Phi(t)]_{ql} [\Phi(t)]_{qk} E\left(\int_0^t[\Phi(s)^{-1}D^{N\mathds{1}}]_l dW_s^0 \int_0^t[\Phi(s)^{-1}D^{N\mathds{1}}]_k dW_s^0\right)\nonumber\\
   &\leq \Sigma_0^2\sum\limits_{l=1}^N \sum\limits_{k=1}^N |[\Phi(t)]_{ql}|\cdot |[\Phi(t)]_{qk}| E\left|\int_0^t[\Phi(s)^{-1}D^{N\mathds{1}}]_l dW_s^0 \int_0^t[\Phi(s)^{-1}D^{N\mathds{1}}]_k dW_s^0\right|\nonumber\\
   &\leq \Sigma_0^2\sum\limits_{l=1}^N \sum\limits_{k=1}^N |[\Phi(t)]_{ql}|\cdot |[\Phi(t)]_{qk}| \left(E\left|\int_0^t[\Phi(s)^{-1}D^{N\mathds{1}}]_l dW_s^0\right|^2\right)^{\frac{1}{2}} \left(E\left|\int_0^t[\Phi(s)^{-1}D^{N\mathds{1}}]_k dW_s^0\right|^2\right)^{\frac{1}{2}}\nonumber\\
   &=\Sigma_0^2\left\{\sum\limits_{l=1}^N |[\Phi(t)]_{ql}| \left(E\left|\int_0^t[\Phi(s)^{-1}D^{N\mathds{1}}]_l dW_s^0\right|^2\right)^{\frac{1}{2}}\right\}^2\nonumber\\
   &=\Sigma_0^2\left\{\sum\limits_{l=1}^N |[\Phi(t)]_{ql}| \left(\int_0^t\left|[\Phi(s)^{-1}D^{N\mathds{1}}]_l\right|^2 ds\right)^{\frac{1}{2}}\right\}^2\nonumber\\
   &\leq \Sigma_0^2\left\{\sum\limits_{l=1}^N |[\Phi(t)]_{ql}| \cdot\max\limits_l \left(\int_0^t\left|[\Phi(s)^{-1}D^{N\mathds{1}}]_l\right|^2 ds\right)^{\frac{1}{2}}\right\}^2\nonumber\\
   &\leq \Sigma_0^2\|\Phi(t)\|_\infty^2 \cdot\max\limits_l \int_0^t\left|[\Phi(s)^{-1}D^{N\mathds{1}}]_l\right|^2 ds\nonumber\\
   &\leq \Sigma_0^2\|\Phi(t)\|_\infty^2 \int_0^t\|\Phi(s)^{-1}D^{N\mathds{1}}\|_\infty^2 ds\nonumber\\
   &\leq \Sigma_0^2\|\Phi(t)\|_\infty^2 \|D^{N\mathds{1}}\|_\infty^2 \int_0^t\|\Phi(s)^{-1}\|_\infty^2 ds\nonumber\\
   &\leq \Sigma_0^2{\Gamma(t)}^2 e^{2|D|t}{E_N}^2 \int_0^t{\Gamma(s)}^{-2}e^{2|D|s}ds\leq C{E_N}^2.
\end{align}
By (\ref{Y_5}) and (\ref{Y_9}), we have
\begin{equation*}
  \begin{split}
  [Y_6]_{qq}&=\left[\Phi(t)\int_0^t\Phi(s)^{-1}(D\cdot D_s^{Nz}-\frac{B^2}{2R}\cdot D_s^{Ng})ds\left(\Phi(t)\int_0^t\Phi(s)^{-1}\Sigma_0D^{N\mathds{1}}dW_s^0\right)^T\right]_{qq}\\
&=\left[\Phi(t)\int_0^t\Phi(s)^{-1}(D\cdot D_s^{Nz}-\frac{B^2}{2R}\cdot D_s^{Ng})ds\right]_q\left[\Phi(t)\int_0^t\Phi(s)^{-1}\Sigma_0D^{N\mathds{1}}dW_s^0\right]_q\\
&\leq\frac{1}{2}\left([Y_5]_{qq}+[Y_9]_{qq}\right).
  \end{split}
\end{equation*}
Then the sixth part $Y_6$ of (\ref{Deltasquare}) satisfies
\begin{equation}\label{EY_6}
  E[Y_6]_{qq}\leq \frac{1}{2}\left(E[Y_5]_{qq}+E[Y_9]_{qq}\right).
\end{equation}
This same bound holds for $Y_8$, i.e. $E[Y_8]_{qq}=E[Y_6]_{qq}$. By (\ref{Deltasquare}), (\ref{EY_1}), (\ref{EY_3}), (\ref{EY_5}), (\ref{EY_2}), (\ref{EY_9}) and (\ref{EY_6}), we derive
\begin{equation}\label{chaos1}
  E|\overline{z}_t^q-\widehat{z}_t^q|^2=E|\Delta_t^q|^2=E[\widetilde{\Delta}_t^N(\widetilde{\Delta}_t^N)^T]_{qq}=\sum\limits_{i=1}^9E[Y_i]_{qq}\leq C\left({E_N}^2+(E_N')^2\right).
\end{equation}

Next, we analysis the second part of the right-hand side of inequality (\ref{triangle}). Denote $e_t^l\triangleq\widehat{x}_t^l-\frac{1}{|\mathcal{C}_l|}\sum\limits_{j\in\mathcal{C}_l}x_t^{oj}$. By the definition of $\widehat{z}_t^q$ and $z_t^{oq}$, we have $\widehat{z}_t^q-z_t^{oq}=\frac{1}{N}\sum\limits_{l=1}^N m_{ql}e_t^l$. From (\ref{xoj}) and (\ref{widetildexl}), $e_t^l$ satisfies
\begin{equation}\label{e_t^l}
  e_t^l=\mu_l-\frac{1}{|\mathcal{C}_l|}\sum\limits_{j\in\mathcal{C}_l}x_0^{j}+\int_0^t\left[\left(A-\frac{B^2}{2R}f_s\right)e_s^l+D\frac{1}{N}\sum\limits_{k=1}^N m_{lk}e_s^k\right]ds-\frac{1}{|\mathcal{C}_l|}\sum\limits_{j\in\mathcal{C}_l}\int_0^t\Sigma dw_s^j.
\end{equation}
Denote $\widetilde{w}_t^l\triangleq\frac{1}{\sqrt{|\mathcal{C}_l|}}\sum\limits_{j\in\mathcal{C}_l}\int_0^tdw_s^j$. By L\'{e}vy's theorem \cite{Karatzas}, $\{\widetilde{w}^l\}_{l=1}^N$ is a $N$-dimensional Brownian motion. Then (\ref{e_t^l}) can be represented as follows
\begin{equation}\label{e_t^lnew}
  e_t^l=\mu_l-\frac{1}{|\mathcal{C}_l|}\sum\limits_{j\in\mathcal{C}_l}x_0^{j}+\int_0^t\left[\left(A-\frac{B^2}{2R}f_s\right)e_s^l+D\frac{1}{N}\sum\limits_{k=1}^N m_{lk}e_s^k\right]ds-\int_0^t\frac{1}{\sqrt{|\mathcal{C}_l|}}\Sigma d\widetilde{w}_s^l.
\end{equation}
Taking square on both sides, we have
\begin{equation*}
  |e_t^l|^2\leq C\left\{\left|\mu_l-\frac{1}{|\mathcal{C}_l|}\sum\limits_{j\in\mathcal{C}_l}x_0^{j}\right|^2+\int_0^t\left[|e_s^l|^2+\frac{1}{N}\sum\limits_{k=1}^N |e_s^k|^2\right]ds+\sup\limits_{0\leq t\leq T}\left|\int_0^t\frac{1}{\sqrt{|\mathcal{C}_l|}}\Sigma d\widetilde{w}_s^l\right|^2\right\}.
\end{equation*}
Applying Gronwall's inequality and the BDG inequality, one has
\begin{equation*}
  \begin{split}
    \frac{1}{N}\sum\limits_{l=1}^N E|e_t^l|^2&\leq C\left\{\frac{1}{N}\sum\limits_{l=1}^NE\left|\mu_l-\frac{1}{|\mathcal{C}_l|}\sum\limits_{j\in\mathcal{C}_l}x_0^{j}\right|^2+\int_0^t \frac{1}{N}\sum\limits_{l=1}^N E|e_s^l|^2ds+\frac{1}{N}\sum\limits_{l=1}^N E\sup\limits_{0\leq t\leq T}\left|\int_0^t\frac{1}{\sqrt{|\mathcal{C}_l|}}\Sigma d\widetilde{w}_s^l\right|^2\right\}\\
    &\leq C\left\{\frac{1}{N}\sum\limits_{l=1}^N\frac{1}{|\mathcal{C}_l|}C_\sigma+\int_0^t \frac{1}{N}\sum\limits_{l=1}^N E|e_s^l|^2ds+\frac{1}{N}\sum\limits_{l=1}^N \frac{1}{|\mathcal{C}_l|}\Sigma^2T\right\}\\
    &\leq C\frac{1}{\min_l|\mathcal{C}_l|}
  \end{split}
\end{equation*}
and so
\begin{equation}\label{chaos2}
  E|\widehat{z}_t^q-z_t^{oq}|^2\leq \frac{1}{N}\sum\limits_{l=1}^N E|e_t^l|^2\leq C\frac{1}{\min_l|\mathcal{C}_l|}.
\end{equation}
It follows from (\ref{triangle}), (\ref{chaos1}) and (\ref{chaos2}) that
\begin{align}\label{overlinez_t^q-z_t^oq}
 \sup\limits_{0\leq t\leq T}E|\overline{z}_t^q-z_t^{oq}|^2&\leq 2\sup\limits_{0\leq t\leq T}E|\overline{z}_t^q-\widehat{z}_t^q|^2+2\sup\limits_{0\leq t\leq T}E|\widehat{z}_t^q-z_t^{oq}|^2\nonumber\\
 &\leq C\left[{E_N}^2+(E_N')^2+\frac{1}{\min_l|\mathcal{C}_l|}\right]
\end{align}
and so the conclusion (i) is true.

Since
\begin{equation*}
  \begin{split}
    E \left||z_t^{oq}|^2-|\overline{z}_t^q|^2\right|&=E\left||z_t^{oq}-\overline{z}_t^q|^2+2\overline{z}_t^q(z_t^{oq}-\overline{z}_t^q)\right|\\
    &\leq E|z_t^{oq}-\overline{z}_t^q|^2+2\left(E|\overline{z}_t^q|^2\right)^{\frac{1}{2}}\left(E|z_t^{oq}-\overline{z}_t^q|^2\right)^{\frac{1}{2}},
  \end{split}
\end{equation*}
Lemma \ref{bounded1} and conclusion (i) lead to the conclusion (ii).

Moreover, by (\ref{overliney^i}) and (\ref{epsilonnashstate}),
$$
x_t^{oi}-\overline{y}_t^i=\int_0^t\left[(A-\frac{B^2}{2R}f_s)(x_s^{oi}-\overline{y}_s^i)+D(z_s^{oq}-\overline{z}_s^q)\right]ds.
$$
Taking square on both sides and then using Gronwall's inequality, we derive the result (iii).

Finally, it follows from (\ref{overliney^i}) that $E\sup\limits_{t\in [0,T]}|\overline{y}_t^i|^2\leq C$, where $C$ is independent of $i$.
Since
$$
E\left||x_t^{oi}|^2-|\overline{y}_t^i|^2\right|\leq E|x_t^{oi}-\overline{y}_t^i|^2+2\left(E|\overline{y}_t^i|^2\right)^{\frac{1}{2}}\left(E|x_t^{oi}
-\overline{y}_t^i|^2\right)^{\frac{1}{2}},
$$
the conclusion (iii) and the boundedness of $\overline{y}^i$ yield the conclusion (iv).
\end{proof}

Now in view of Lemmas \ref{bounded1} and \ref{bounded2} and Proposition \ref{error1}, we obtain the following approximation result.
\begin{proposition}\label{error2}
For $i\in\{1,\ldots,K\}$, it holds that $\left|J_i(u^{oi},u^{-oi})-J_i^*(\overline{u}^i)\right|=O({\delta_K}^\frac{1}{2})$.
\end{proposition}
\begin{proof}
Denote $\nu_t^{oi}=H(z_t^{oi}+\eta)$. Then we have
\begin{equation*}
  \begin{split}
    &\left|J_i(u^{oi},u^{-oi})-J_i^*(\overline{u}^i)\right|\\
    &\leq E\int_0^T\left[Q|(x_t^{oi}-\nu_t^{oi})^2-(\overline{y}_t^i-\overline{\nu}_t^q)^2|+R|(u_t^{oi})^2-(\overline{u}_t^i)^2|\right]dt+Q_TE|(x_T^{oi}-\nu_T^{oi})^2-(\overline{y}_T^i-\overline{\nu}_T^q)^2|\\
    &\leq(QT+Q_T)\sup\limits_t E|(x_t^{oi}-\nu_t^{oi})^2-(\overline{y}_t^i-\overline{\nu}_t^q)^2|+RT\sup\limits_t E|(u_t^{oi})^2-(\overline{u}_t^i)^2|\\
    &\leq(QT+Q_T)\sup\limits_t E\left||x_t^{oi}|^2-|\overline{y}_t^i|^2+|\nu_t^{oi}|^2-|\overline{\nu}_t^q|^2-2(x_t^{oi}\nu_t^{oi}-\overline{y}_t^i\nu_t^{oi}+\overline{y}_t^i\nu_t^{oi}-\overline{y}_t^i\overline{\nu}_t^q)\right|\\
    &\quad +RT\sup\limits_t E\left[\frac{B^2}{4R^2}f_t^2\left||x_t^{oi}|^2-|\overline{y}_t^i|^2\right|+\frac{B^2}{2R^2}|f_t|\cdot|\overline{g}_t^q|\cdot|x_t^{oi}-\overline{y}_t^i|\right]\\
    &\leq(QT+Q_T)\sup\limits_t E\left[\left||x_t^{oi}|^2-|\overline{y}_t^i|^2\right|+H^2\left||z_t^{oq}|^2-|\overline{z}_t^q|^2\right|+2H^2|\eta|\cdot|z_t^{oq}-\overline{z}_t^q|+2|H|\cdot|z_t^{oq}+\eta|\cdot|x_t^{oi}-\overline{y}_t^i|\right.\\
    &\left.\quad +2|H|\cdot |\overline{y}_t^i|\cdot|z_t^{oq}-\overline{z}_t^q|\right]+RT\sup\limits_t E\left[\frac{B^2}{4R^2}f_t^2\left||x_t^{oi}|^2-|\overline{y}_t^i|^2\right|+\frac{B^2}{2R^2}|f_t|\cdot|\overline{g}_t^q|\cdot|x_t^{oi}-\overline{y}_t^i|\right]\\
    &\leq C\left[\sup\limits_tE \left||x_t^{oi}|^2-|\overline{y}_t^i|^2\right|+\sup\limits_tE \left||z_t^{oq}|^2-|\overline{z}_t^q|^2\right|+\left(\sup\limits_tE\left|z_t^{oq}-\overline{z}_t^q\right|^2\right)^{\frac{1}{2}}+\left(\sup\limits_tE\left|x_t^{oi}-\overline{y}_t^i\right|^2\right)^{\frac{1}{2}}\right].
  \end{split}
\end{equation*}
Then we have the desired result.
\end{proof}

Next we introduce two perturbed systems for any $i\in\mathcal{C}_q$: a perturbed closed-loop system and a perturbed limiting system. Let us begin with the perturbed closed-loop system.

For the agent $\mathcal{A}_i$, consider an admissible alternative control $v^i\in\mathcal{U}$ and introduce the corresponding dynamics
\begin{equation}\label{widetildex^i}
  \left\{
  \begin{split}
    &d\widetilde{x}_t^i=(A\widetilde{x}_t^i+Bv_t^i+D\widetilde{z}_t^i)dt+\Sigma dw_t^i+\Sigma_0 dW_t^0,\\
    &\widetilde{x}_0^i=x_0^i,
  \end{split}
  \right.
\end{equation}
where $\widetilde{z}_t^i\triangleq\frac{1}{N}\sum_{l=1}^N m_{ql}\frac{1}{|\mathcal{C}_l|}\sum_{j\in\mathcal{C}_l}\widetilde{x}_t^j$, and other agents keep the control $u^{oj}$, $j\neq i$, i.e.
\begin{equation}\label{widetildex^j}
  \left\{
  \begin{split}
    &d\widetilde{x}_t^j=\left[(A-\frac{B^2}{2R}f_t)\widetilde{x}_t^j-\frac{B^2}{2R}\overline{g}_t^l+D\widetilde{z}_t^j\right]dt+\Sigma dw_t^j+\Sigma_0 dW_t^0,\\
    &\widetilde{x}_0^j=x_0^j.
  \end{split}
  \right.
\end{equation}
Since $\widetilde{z}_t^j=\widetilde{z}_t^k$ for any $j,k\in \mathcal{C}_l$, denote $\widetilde{z}_t^l=\widetilde{z}_t^j$ for all $j\in \mathcal{C}_l$. The cost function of $\mathcal{A}_i$ is given by
\begin{equation*}
  J_i(v^i,u^{-oi})=E\left\{\int_0^T\left[Q(\widetilde{x}_t^i-\widetilde{\nu}_t^i)^2+R(v_t^i)^2\right]dt+Q_T(\widetilde{x}_T^i-\widetilde{\nu}_T^i)^2\right\},\quad \widetilde{\nu}_t^i\triangleq H(\widetilde{z}_t^i+\eta).
\end{equation*}

By (\ref{y^i}), considering the admissible alternative control $v^i\in\mathcal{U}$, we also introduce the perturbed limiting system of agent $\mathcal{A}_i$ as follows:
\begin{equation*}
\left\{
  \begin{split}
     &d\widetilde{y}_t^i=(A\widetilde{y}_t^i+Bv_t^i+D\overline{z}_t^q)dt+\Sigma dw_t^i+\Sigma_0dW_t^0, \\
      &\widetilde{y}_0^i=x_0^i.
  \end{split}
  \right.
\end{equation*}
The corresponding cost function of $\mathcal{A}_i$ is given by
\begin{equation*}
  J_i^*(v^i)=E\left\{\int_0^T\left[Q(\widetilde{y}_t^i-\overline{\nu}_t^q)^2+R(v_t^i)^2\right]dt+Q_T(\widetilde{y}_T^i-\overline{\nu}_T^q)^2\right\},\quad \overline{\nu}_t^q\triangleq H(\overline{z}_t^q+\eta).
\end{equation*}

Now we give some estimation results for $\widetilde{z}_t^l$ and $\widetilde{x}_t^i$.
\begin{lemma}\label{bounded3} For the dynamics (\ref{widetildex^i}) and (\ref{widetildex^j}), it holds that
$$\max\limits_{1\leq l\leq N}\sup\limits_{0\leq t\leq T}E|\widetilde{z}_t^l|^2\leq C\left(1+\frac{1}{N|\mathcal{C}_q|}E\int_0^T|v_t^i|^2ds\right),\quad \sup\limits_{0\leq t\leq T}E|\widetilde{x}_t^i|^2\leq C\left(1+E\int_0^T|v_t^i|^2ds\right),$$
where the constant $C$ in second conclusion is independent of $i$.
\end{lemma}
\begin{proof}
Taking square on both sides of (\ref{widetildex^i}), we have
\begin{equation*}
 |\widetilde{x}_t^i|^2\leq C\left[|x_0^i|^2+\int_0^t \left(|\widetilde{x}_s^i|^2+|v_s^i|^2+|\widetilde{z}_s^q|^2\right)ds +\left|\int_0^t\Sigma dw_s^i\right|^2+\left|\int_0^t\Sigma_0 dW_s^0\right|^2\right].
\end{equation*}
For any $j\in\mathcal{C}_q$ with $j\neq i$, taking square on both sides of (\ref{widetildex^j}), we derive
\begin{equation*}
 |\widetilde{x}_t^j|^2\leq C\left[|x_0^j|^2+\int_0^t \left(|\widetilde{x}_s^j|^2+|\overline{g}_s^q|^2+|\widetilde{z}_s^q|^2\right)ds +\left|\int_0^t\Sigma dw_s^j\right|^2+\left|\int_0^t\Sigma_0 dW_s^0\right|^2\right].
\end{equation*}
Then
\begin{align}\label{sumwidetildexq}
  \frac{1}{|\mathcal{C}_q|}\sum\limits_{j\in\mathcal{C}_q}|\widetilde{x}_t^j|^2\leq & C\left\{\frac{1}{|\mathcal{C}_q|}\sum\limits_{j\in\mathcal{C}_q}|x_0^j|^2+\int_0^t \left[\frac{1}{|\mathcal{C}_q|}\sum\limits_{j\in\mathcal{C}_q}|\widetilde{x}_s^j|^2+|\overline{g}_s^q|^2+|\widetilde{z}_s^q|^2+\frac{1}{|\mathcal{C}_q|}\left(|v_s^i|^2-|\overline{g}_s^q|^2\right)\right]ds\right.\nonumber\\
  &\left.+\frac{1}{|\mathcal{C}_q|}\sum\limits_{j\in\mathcal{C}_q}\left|\int_0^t\Sigma dw_s^j\right|^2+\left|\int_0^t\Sigma_0 dW_s^0\right|^2\right\}.
\end{align}
For any $j\in\mathcal{C}_l$ with $l\neq q$, taking square and summation of $j\in\mathcal{C}_l$ on both sides of (\ref{widetildex^j}), we get
\begin{align}\label{sumwidetildexl}
  \frac{1}{|\mathcal{C}_l|}\sum\limits_{j\in\mathcal{C}_l}|\widetilde{x}_t^j|^2\leq & C\left[\frac{1}{|\mathcal{C}_l|}\sum\limits_{j\in\mathcal{C}_l}|x_0^j|^2+\int_0^t \left(\frac{1}{|\mathcal{C}_l|}\sum\limits_{j\in\mathcal{C}_l}|\widetilde{x}_s^j|^2+|\overline{g}_s^l|^2+|\widetilde{z}_s^l|^2\right)ds\right.\nonumber\\
  &\left.+\frac{1}{|\mathcal{C}_l|}\sum\limits_{j\in\mathcal{C}_l}\left|\int_0^t\Sigma dw_s^j\right|^2+\left|\int_0^t\Sigma_0 dW_s^0\right|^2\right].
\end{align}
By (\ref{sumwidetildexq}) and (\ref{sumwidetildexl}),
\begin{equation*}
\begin{split}
  \frac{1}{N}\sum\limits_{k=1}^N\frac{1}{|\mathcal{C}_k|}\sum\limits_{n\in\mathcal{C}_k}|\widetilde{x}_t^n|^2\leq & C\left\{\frac{1}{N}\sum\limits_{k=1}^N\frac{1}{|\mathcal{C}_k|}\sum\limits_{n\in\mathcal{C}_k}|\widetilde{x}_0^n|^2+\int_0^t \left[\frac{1}{N}\sum\limits_{k=1}^N\frac{1}{|\mathcal{C}_k|}\sum\limits_{n\in\mathcal{C}_k}|\widetilde{x}_s^n|^2+\frac{1}{N}\sum\limits_{k=1}^N|\overline{g}_s^k|^2+\frac{1}{N}\sum\limits_{k=1}^N|\widetilde{z}_s^k|^2\right.\right.\\
  &\left.\left.+\frac{1}{N|\mathcal{C}_q|}\left(|v_s^i|^2+|\overline{g}_s^q|^2\right)\right]ds+\frac{1}{N}\sum\limits_{k=1}^N\frac{1}{|\mathcal{C}_k|}\sum\limits_{n\in\mathcal{C}_k}\left|\int_0^t\Sigma dw_s^n\right|^2+\left|\int_0^t\Sigma_0 dW_s^0\right|^2\right\}.
\end{split}
\end{equation*}
Taking expectation on both sides and applying BDG inequality, we have
\begin{equation*}
   E\frac{1}{N}\sum\limits_{k=1}^N\frac{1}{|\mathcal{C}_k|}\sum\limits_{n\in\mathcal{C}_k}|\widetilde{x}_t^n|^2\leq C\left[1+\int_0^tE \frac{1}{N}\sum\limits_{k=1}^N\frac{1}{|\mathcal{C}_k|}\sum\limits_{n\in\mathcal{C}_k}|\widetilde{x}_s^n|^2ds+\frac{1}{N|\mathcal{C}_q|}E\int_0^T\left(|v_s^i|^2+|\overline{g}_s^q|^2\right)ds\right].
\end{equation*}
By Lemma \ref{bounded1} and applying Gronwall's inequality, for any $l\in\{1,\ldots,N\}$, we have
\begin{equation*}
\begin{split}
 E|\widetilde{z}_t^l|^2&\leq E\frac{1}{N}\sum\limits_{k=1}^N\frac{1}{|\mathcal{C}_k|}\sum\limits_{n\in\mathcal{C}_k}|\widetilde{x}_t^n|^2\leq C\left[1+\frac{1}{N|\mathcal{C}_q|}E\int_0^T\left(|v_s^i|^2+|\overline{g}_s^q|^2\right)ds\right]\\
  &\leq C\left(1+\frac{1}{N|\mathcal{C}_q|}E\int_0^T|v_s^i|^2ds\right).
\end{split}
\end{equation*}
Then we derive the first result. The other result can be proved similar to Lemma \ref{bounded2}.
\end{proof}
Since the parameters of cost function is positive (nonnegative),
$$
E\int_0^TR|v_s^i|^2ds\leq J_i(v^i,u^{-oi}),
$$
it is enough to consider the control $v^i$ satisfying
\begin{equation}\label{boundv^i}
  E\int_0^T|v_s^i|^2ds\leq \frac{1}{R}\left[J_i^*(\overline{u}^i)+O({\delta_K}^{\frac{1}{2}})\right].
\end{equation}
Otherwise, applying Proposition \ref{error2} yields
$$
E\int_0^TR|v_s^i|^2ds>J_i^*(\overline{u}^i)+O({\delta_K}^{\frac{1}{2}})\geq J_i(u^{oi},u^{-oi}),
$$
hence the $\epsilon$-Nash property holds trivially.

The following proposition presents the approximation between the perturbed limiting system and the perturbed closed-loop system for agent $\mathcal{A}_i$.
\begin{proposition}\label{error3}
For fixed $i\in\mathcal{C}_q$, it holds that
\begin{enumerate}[(i)]
\item $\sup\limits_{0\leq t\leq T}E\left|\widetilde{z}_t^q-\overline{z}_t^q\right|^2=O(\delta_K)$;
\item $\sup\limits_{0\leq t\leq T}E \left||\widetilde{z}_t^q|^2-|\overline{z}_t^q|^2\right|=O({\delta_K}^{\frac{1}{2}})$;
\item $\sup\limits_{0\leq t\leq T}E\left|\widetilde{x}_t^i-\widetilde{y}_t^i\right|^2=O(\delta_K)$;
\item $\sup\limits_{0\leq t\leq T}E \left||\widetilde{x}_t^i|^2-|\widetilde{y}_t^i|^2\right|=O({\delta_K}^{\frac{1}{2}})$.
\end{enumerate}
\end{proposition}
\begin{proof}
Denote $\bar{\bar{x}}_t^l\triangleq \frac{1}{|\mathcal{C}_l|}\sum_{j\in\mathcal{C}_l}\widetilde{x}_t^j$, $l\in\{1,\ldots,N\}$. For each $l\neq q$,
\begin{equation}\label{barbarx^l}
  \left\{
  \begin{split}
   &d\bar{\bar{x}}_t^l=\left[(A-\frac{B^2}{2R}f_t)\bar{\bar{x}}_t^l
   -\frac{B^2}{2R}\overline{g}_t^l+D\widetilde{z}_t^l\right]dt
   +\Sigma\frac{1}{|\mathcal{C}_l|}\sum\limits_{j\in\mathcal{C}_l}dw_t^j+\Sigma_0dW_t^0,\\
   &\bar{\bar{x}}_0^l=\frac{1}{|\mathcal{C}_l|}\sum_{j\in\mathcal{C}_l}x_0^j
  \end{split}
  \right.
\end{equation}
and
\begin{equation}\label{barbarx^q}
  \left\{
  \begin{split}
  &d\bar{\bar{x}}_t^q=\left[(A-\frac{B^2}{2R}f_t)\bar{\bar{x}}_t^q-\frac{B^2}{2R}\overline{g}_t^q
  +D\widetilde{z}_t^q+\frac{1}{|\mathcal{C}_q|}\left(Bv_t^i+\frac{B^2}{2R}f_t\widetilde{x}_t^i
  +\frac{B^2}{2R}\overline{g}_t^q\right)\right]dt+\Sigma\frac{1}{|\mathcal{C}_q|}\sum\limits_{j\in\mathcal{C}_q}dw_t^j
  +\Sigma_0dW_t^0,\\
   &\bar{\bar{x}}_0^q=\frac{1}{|\mathcal{C}_q|}\sum_{j\in\mathcal{C}_q}x_0^j.
  \end{split}
  \right.
\end{equation}
Recall that for $l\in\{1,\ldots,N\}$,
$$
\widetilde{z}_t^l=\frac{1}{N}\sum_{k=1}^N m_{lk}\frac{1}{|\mathcal{C}_k|}\sum_{n\in\mathcal{C}_k}\widetilde{x}_t^n=\frac{1}{N}\sum_{k=1}^N m_{lk}\bar{\bar{x}}_t^k.
$$
By (\ref{barbarx^l}) and (\ref{barbarx^q}), we derive
\begin{equation}\label{widetildeztl}
  \left\{
  \begin{split}
  &d\widetilde{z}_t^l=\left[(A-\frac{B^2}{2R}f_t)\widetilde{z}_t^l-\frac{B^2}{2R}\frac{1}{N}\sum_{k=1}^N m_{lk}\overline{g}_t^k+D\frac{1}{N}\sum_{k=1}^N m_{lk}\widetilde{z}_t^k+\frac{m_{lq}}{N|\mathcal{C}_q|}\left(Bv_t^i+\frac{B^2}{2R}f_t\widetilde{x}_t^i+\frac{B^2}{2R}\overline{g}_t^q\right)\right]dt\\
  &\qquad +\Sigma\frac{1}{N}\sum_{k=1}^N m_{lk}\frac{1}{|\mathcal{C}_k|}\sum\limits_{n\in\mathcal{C}_k}dw_t^n+\Sigma_0\frac{1}{N}\sum_{k=1}^N m_{lk}dW_t^0\\
   &\widetilde{z}_0^l=\frac{1}{N}\sum_{k=1}^N m_{lk}\frac{1}{|\mathcal{C}_k|}\sum_{n\in\mathcal{C}_k}x_0^n.
  \end{split}
  \right.
\end{equation}
From (\ref{xoj}), one has
\begin{equation}\label{ztol}
  \left\{
  \begin{split}
  &dz_t^{ol}=\left[(A-\frac{B^2}{2R}f_t)z_t^{ol}-\frac{B^2}{2R}\frac{1}{N}\sum_{k=1}^N m_{lk}\overline{g}_t^k+D\frac{1}{N}\sum_{k=1}^N m_{lk}z_t^{ok}\right]dt\\
  &\qquad +\Sigma\frac{1}{N}\sum_{k=1}^N m_{lk}\frac{1}{|\mathcal{C}_k|}\sum\limits_{n\in\mathcal{C}_k}dw_t^n+\Sigma_0\frac{1}{N}\sum_{k=1}^N m_{lk}dW_t^0,\\
   &z_0^{ol}=\frac{1}{N}\sum_{k=1}^N m_{lk}\frac{1}{|\mathcal{C}_k|}\sum_{n\in\mathcal{C}_k}x_0^n.
  \end{split}
  \right.
\end{equation}
Then it follows from (\ref{widetildeztl}) and (\ref{ztol}) that
\begin{equation*}
  \widetilde{z}_t^l-z_t^{ol}=\int_0^t\left[(A-\frac{B^2}{2R}f_s)(\widetilde{z}_s^l-z_s^{ol})+D\frac{1}{N}\sum_{k=1}^N m_{lk}(\widetilde{z}_s^k-z_s^{ok})+\frac{m_{lq}}{N|\mathcal{C}_q|}\left(Bv_s^i+\frac{B^2}{2R}f_s\widetilde{x}_s^i
  +\frac{B^2}{2R}\overline{g}_s^q\right)\right]ds.
\end{equation*}
Taking square and summation over $l$ from $1$ to $N$ on both sides, we have
\begin{equation*}
  \sum_{l=1}^N|\widetilde{z}_t^l-z_t^{ol}|^2\leq C\int_0^t \left[\sum_{l=1}^N|\widetilde{z}_s^l-z_s^{ol}|^2+\frac{1}{N|\mathcal{C}_q|^2}\left(|v_s^i|^2+|\widetilde{x}_s^i|^2+|\overline{g}_s^q|^2\right)\right]ds
\end{equation*}
Taking expectation on both sides and applying Gronwall's inequality, one has
\begin{equation}\label{++}
  E\sum_{l=1}^N|\widetilde{z}_t^l-z_t^{ol}|^2\leq C\frac{1}{N|\mathcal{C}_q|^2}E\int_0^T\left(|v_s^i|^2+|\widetilde{x}_s^i|^2+|\overline{g}_s^q|^2\right)ds\leq C\frac{1}{N|\mathcal{C}_q|^2}\left(1+E\int_0^T|v_s^i|^2ds\right).
\end{equation}
It is easy to check that $J_i^*(\overline{u}^i)$ is bounded uniformly with respect to $i$ and so \eqref{boundv^i} and \eqref{++} lead to
\begin{equation}\label{widetildez_t^q-z_t^oq}
  \sup_{0\leq t\leq T}E|\widetilde{z}_t^q-z_t^{oq}|^2\leq C\frac{1}{N|\mathcal{C}_q|^2}\left(1+E\int_0^T|v_s^i|^2ds\right)\leq C\frac{1}{N|\mathcal{C}_q|^2}.
\end{equation}
By (\ref{overlinez_t^q-z_t^oq}) and (\ref{widetildez_t^q-z_t^oq}), and the triangle inequality, one has
\begin{equation*}
  \begin{split}
    \sup_{0\leq t\leq T}E\left|\widetilde{z}_t^q-\overline{z}_t^q\right|^2&\leq 2\sup_{0\leq t\leq T}E\left|\widetilde{z}_t^q-z_t^{oq}\right|^2+2\sup_{0\leq t\leq T}E\left|z_t^{oq}-\overline{z}_t^q\right|^2\\
    &\leq C\left(\frac{1}{N|\mathcal{C}_q|^2}+{E_N}^2+(E_N')^2+\frac{1}{\min_l|\mathcal{C}_l|}\right)\\
    &\leq C\left({E_N}^2+(E_N')^2+\frac{1}{\min_l|\mathcal{C}_l|}\right)\\
    &=C\delta_K.
  \end{split}
\end{equation*}
Then the conclusion (i) holds. The rest conclusions (ii)-(iv) can be obtained similar to the proof of Proposition \ref{error1}.
\end{proof}
By Lemma \ref{bounded3} and Proposition \ref{error3}, we can deduce the following proposition.
\begin{proposition}\label{error4}
For any $v^i\in\mathcal{U}$ with $i\in\{1,\ldots,K\}$, it holds that
$$
\left|J_i(v^i,u^{-oi})-J_i^*(v^i)\right|=O\left({\delta_K}^{\frac{1}{2}}\right).
$$
\end{proposition}
\begin{proof}
The proof is similar to the one of Proposition \ref{error2} and so we omit it here.
\end{proof}

Thus, we can conclude the main result in this section as follows.
\begin{theorem}\label{epsilonnashtheorem}
Under Assumptions \ref{finiteeigenvalues}, \ref{monotonicity} (or \ref{assumptionriccati}), \ref{assumptionmu} and \ref{approximation}, the set of strategies $(u^{o1},\ldots,u^{oK})$ given by (\ref{epsilonnashstrategy}) is an $\epsilon$-Nash equilibrium, with $\epsilon=O({\delta_K}^{\frac{1}{2}})$, where $\delta_K={E_N}^2+(E_N')^2+\frac{1}{\min_l|\mathcal{C}_l|}\rightarrow 0$, as $N\rightarrow\infty$, $\min_l|\mathcal{C}_l|\rightarrow\infty$.
\end{theorem}
\begin{proof}
For $i\in\{1,\ldots,K\}$, by Propositions \ref{error2} and \ref{error4} and the optimality of $\overline{u}^i$, we derive
\begin{equation*}
  \begin{split}
    J_i(u^{oi},u^{-oi})\leq& J_i^*(\overline{u}^i)+O({\delta_K}^{\frac{1}{2}})\leq J_i^*(v^i)+O({\delta_K}^{\frac{1}{2}})\\
    \leq& J_i(v^i,u^{-oi})+O({\delta_K}^{\frac{1}{2}}),
  \end{split}
\end{equation*}
which yields the result.
\end{proof}

\begin{remark}
In this paper, we assume that the graphon operators are of finite rank. It's noted that for graphon operators with infinite rank, their eigenvalues accumulate at zero \cite{Lovasz}.  Thus, for some infinite rank graphons, we try to discuss whether approximate results for the Nash equilibrium of Problem \ref{Knash} can be obtained by using finite number of their eigenvalues. We provide a preliminary discussion of this issue through the following example.
\end{remark}
\vspace{1ex}
\begin{example}
We consider the uniform attachment graphon $M(\alpha,\beta)=1-\max(\alpha,\beta)$ with $\alpha,\beta\in[0,1]$. The $N$-node graph with stepfunction type graphon $M^{[N]}$ satisfies Assumption \ref{approximation}, i.e., $\max_{q\in\{1\ldots,N\}}N\|(M-M^{[N]})\mathds{1}_{P_q}\|_1\rightarrow 0$, as $N\rightarrow\infty$. By \cite[Prop 10]{Gao5}, the eigen pairs $(f_l,\lambda_l)$, $l=1,2,3,\cdots$ of $M$ are given by $\big(\sqrt{2}\cos(\frac{k\pi(\cdot)}{2}),\frac{4}{k^2\pi^2}\big)$, $k=1,3,5,\cdots$. Let $$\widetilde{M}(\alpha,\beta)\triangleq \sum\limits_{k=1,3,5,7,9}\frac{4}{k^2\pi^2}\sqrt{2}\cos(\frac{k\pi\alpha}{2})\sqrt{2}\cos(\frac{k\pi\beta}{2}),$$ it's easy to check that $\big(\sqrt{2}\cos(\frac{k\pi(\cdot)}{2}),\frac{4}{k^2\pi^2}\big)$, $k=1,3,5,7,9$, are eigen pairs of $\widetilde{M}$, and rank $\widetilde{M}=5$. Using the triangle inequality yields
\begin{align}\label{unknownx}
\max_{q\in\{1\ldots,N\}} & N\|(\widetilde{M}-M^{[N]})\mathds{1}_{P_q}\|_1\\
& \leq \max_{q\in\{1\ldots,N\}}N\|(\widetilde{M}-M)\mathds{1}_{P_q}\|_1+\max_{q\in\{1\ldots,N\}}N\|(M-M^{[N]})\mathds{1}_{P_q}\|_1.
\end{align}
To avoid confusion with previous notations, let $\widetilde{G}\triangleq \max_{q\in\{1\ldots,N\}}N\|(\widetilde{M}-M)\mathds{1}_{P_q}\|_1$, $G_N\triangleq \max_{q\in\{1\ldots,N\}}N\|(M-M^{[N]})\mathds{1}_{P_q}\|_1$. Then $G_N\rightarrow 0$ as $N\rightarrow \infty$. Using the same construction of strategies in (\ref{epsilonnashstrategy}), we design the feedback strategy for agent $\mathcal{A}_i$ in $\mathcal{C}_q$ as follows
\begin{equation}
	\left\lbrace
	\begin{split}
		& \tilde{u}_t^{oi}=-\frac{B}{2R}(f_t\tilde{x}_t^{oi}+\tilde{g}_t^q),\\
		&\tilde{g}_t^q\triangleq \sum_{l=1}^5 g_t^l N\int_{P_q} f_l(\alpha)d\alpha+\mathring{g}_t N\int_{P_q}\Big(\mathds{1}-\sum\limits_{l=1}^5 \langle \mathds{1},f_l\rangle f_l\Big)(\alpha)d\alpha,
		\end{split}
		\right.
\end{equation}
where $\{f_l\}_{l=1}^5$ are eigenfunctions of $\widetilde{M}$, $g^l$ and $\mathring{g}$ are given by (\ref{FBSDEl}) and (\ref{mathringg}) respectively. From the discussion in this section, using (\ref{unknownx}) yields $J_i(\tilde{u}^{oi},\tilde{u}^{-oi})\leq J_i(v^i,\tilde{u}^{-oi})+O({\widetilde{\delta}_K}^{\frac{1}{2}})$, where $\widetilde{\delta}_K=(\widetilde{G}+G_N)^2+(E_N')^2+\frac{1}{\min_l|\mathcal{C}_l|}$. Next we calculate the error $\widetilde{G}$ below
\begin{align*}
& N\|(\widetilde{M}-M)\mathds{1}_{P_q}\|_1\\
= & N\int_0^1\Big|\int_{P_q}\sum_{k=11,13,\cdots} \frac{8}{k^2\pi^2}\cos(\frac{k\pi\alpha}{2})\cos(\frac{k\pi\beta}{2})d\beta\Big|d\alpha\\
\leq & N\int_0^1\sum_{k=11,13,\cdots} \frac{8}{k^2\pi^2} \Big|\cos(\frac{k\pi\alpha}{2})\Big|\cdot \Big|\int_{P_q}\cos(\frac{k\pi\beta}{2})d\beta\Big|d\alpha\\
= & N\sum_{k=11,13,\cdots} \frac{16}{k^3\pi^3} \Big|\sin(\frac{k\pi q}{2N})-\sin(\frac{k\pi (q-1)}{2N})\Big|\cdot \int_0^1\Big|\cos(\frac{k\pi\alpha}{2})\Big|d\alpha.
\end{align*}
Since
$$
\int_0^1\Big|\cos(\frac{k\pi\alpha}{2})\Big|d\alpha=\int_0^\frac{1}{k}\Big|\cos(\frac{k\pi\alpha}{2})\Big|d\alpha+\frac{k-1}{2}\int_\frac{1}{k}^\frac{3}{k}\Big|\cos(\frac{k\pi\alpha}{2})\Big|d\alpha=\frac{2}{\pi},
$$
we have
\begin{align*}
N\|(\widetilde{M}-M)\mathds{1}_{P_q}\|_1 & \leq N\sum_{k=11,13,\cdots} \frac{32}{k^3\pi^4} \Big|\sin(\frac{k\pi q}{2N})-\sin(\frac{k\pi (q-1)}{2N})\Big|\\
& =N\sum_{k=11,13,\cdots} \frac{64}{k^3\pi^4}\Big|\cos(\frac{k\pi (2q-1)}{4N})\Big|\cdot \Big|\sin(\frac{k\pi}{4N})\Big|\\
& \leq \sum_{k=11,13,\cdots}\frac{16}{k^2\pi^3}=\frac{16}{\pi^3}\Big(\frac{\pi^2}{8}-\sum_{k=1,3,5,7,9}\frac{1}{k^2}\Big)\approx 0.0257
\end{align*}
Then, one has $\widetilde{G}\leq 0.0258$. Since $G_N,E_N',\frac{1}{\min_l|\mathcal{C}_l|}\rightarrow 0$ as $N\rightarrow\infty$, $\min_l|\mathcal{C}_l|\rightarrow\infty$, it follows that ${\widetilde{\delta}_K}^{\frac{1}{2}}\leq 0.0258$ when $N$ and $\min_l|\mathcal{C}_l|$ are both large enough. Therefore, it is possible to obtain an approximation of the Nash equilibrium of Problem \ref{Knash} by choosing rank of $\widetilde{M}$ large enough and suitable parameters in the model. More advanced results are reserved for future work.
	
\end{example}

Finally, we provide several examples of a sequence of graphs and the limit graphon, which satisfy the corresponding assumptions.
\begin{example}[\cite{Gao4}]\label{graphs}
Consider a sinusoidal graphon $M(\alpha,\beta)=-\cos(2\pi(\alpha-\beta))$ with $\alpha,\beta\in[0,1]$, which has simple spectral characterizations (see \cite{Gao4}). The eigenfunctions of graphon $M$ associated with nonzero eigenvalues are $\sqrt{2}\cos(2\pi(\cdot))$ and $\sqrt{2}\sin(2\pi(\cdot))$, and the corresponding eigenvalues are $-\frac{1}{2}$, $-\frac{1}{2}$. Then Assumption \ref{finiteeigenvalues} holds. Moreover, we generate a sequence of graphs of size $N$ from the graphon $M$, by connecting nodes $i$ and $j$ with weight $M(\frac{i-1}{N},\frac{j-1}{N})$, for $i,j\in\{1,\ldots,N\}$, which means the adjacency matrices $M_N=[m_{ij}]=[M(\frac{i-1}{N},\frac{j-1}{N})]$. This sequence of graphs and the graphon $M$ satisfy Assumption \ref{approximation}. Indeed, letting $M^{[N]}$ denote the step function type graphons corresponding to $M_N$, we have that for each $q\in\{1,\ldots,N\}$,
\begin{align}\label{graphs1}
 \left((M-M^{[N]})\mathds{1}_{P_q}\right)(\alpha)=&\int_0^1(M-M^{[N]})(\alpha,\beta)\mathds{1}_{P_q}(\beta)d\beta\nonumber\\
 =&-\int_0^1\left[\cos(2\pi(\alpha-\beta))-\sum_{i=1}^N\sum_{j=1}^N\mathds{1}_{P_i}(\alpha)\mathds{1}_{P_j}(\beta)\cos(2\pi\frac{i-j}{N})\right]\mathds{1}_{P_q}(\beta)d\beta\nonumber\\
 =&-\int_{P_q}\left[\cos(2\pi(\alpha-\beta))-\sum_{i=1}^N\mathds{1}_{P_i}(\alpha)\cos(2\pi\frac{i-q}{N})\right]d\beta.
\end{align}
Applying the mean value theorem for integrals, there exists $\xi\in[\frac{q-1}{N},\frac{q}{N}]$ such that
\begin{align}\label{graphs2}
  &-\int_{P_q}\left[\cos(2\pi(\alpha-\beta))-\sum_{i=1}^N\mathds{1}_{P_i}(\alpha)\cos(2\pi\frac{i-q}{N})\right]d\beta\nonumber\\
  =&-\frac{1}{N}\left[\cos(2\pi(\alpha-\xi))-\sum_{i=1}^N\mathds{1}_{P_i}(\alpha)\cos(2\pi\frac{i-q}{N})\right]\nonumber\\
  =&-\frac{1}{N}\sum_{i=1}^N\mathds{1}_{P_i}(\alpha)\left[\cos(2\pi(\alpha-\xi))-\cos(2\pi\frac{i-q}{N})\right].
\end{align}
Combining (\ref{graphs1}) and (\ref{graphs2}), we have
\begin{equation}\label{graphs3}
  N\int_0^1\left|\left((M-M^{[N]})\mathds{1}_{P_q}\right)(\alpha)\right|d\alpha\leq\int_0^1\sum_{i=1}^N\mathds{1}_{P_i}(\alpha)\left|\cos(2\pi(\alpha-\xi))-\cos(2\pi\frac{i-q}{N})\right|d\alpha.
\end{equation}
For $\alpha\in [\frac{i-1}{N},\frac{i}{N}]$, one has
\begin{align}\label{graphs4}
  \left|\cos(2\pi(\alpha-\xi))-\cos(2\pi\frac{i-q}{N})\right|=& 2\left|\sin(\pi(\alpha-\xi+\frac{i-q}{N}))\right|\cdot\left|\sin(\pi(\alpha-\xi-\frac{i-q}{N}))\right|\nonumber\\
  \leq & 2\left|\sin(\pi(\alpha-\xi-\frac{i-q}{N}))\right|\nonumber\\
  \leq & 2\pi\left|\alpha-\xi-\frac{i-q}{N}\right|\nonumber\\
  \leq &\frac{4\pi}{N}.
\end{align}
It follows from (\ref{graphs3}) and (\ref{graphs4}) that
$$
N\int_0^1\left|\left((M-M^{[N]})\mathds{1}_{P_q}\right)(\alpha)\right|d\alpha\leq \frac{4\pi}{N}
$$
and so
$$
\max\limits_{q\in\{1\ldots,N\}}N\|(M-M^{[N]})\mathds{1}_{P_q}\|_1=\max\limits_{q\in\{1\ldots,N\}}N\int_0^1\left|\left((M-M^{[N]})\mathds{1}_{P_q}\right)(\alpha)\right|d\alpha\leq \frac{4\pi}{N}.
$$
Thus, the given $M^{[N]}$ and $M$ satisfy Assumption \ref{approximation}.

\end{example}

\begin{example}[\cite{Parise}]\label{secondgraphs}
For any $a\in[-1,1]$, consider a graphon $M(\alpha,\beta)=a v(\alpha)v(\beta)$ with $\alpha,\beta\in[0,1]$, where $v(\alpha)=\frac{1}{\sqrt{2}(\alpha+\frac{1}{2})^\frac{1}{4}}$. This graphon is rank $1$ with eigenvalue $\lambda=a\|v\|_2^2$ and eigenfunction $\frac{v}{\|v\|_2}$. Then Assumption \ref{finiteeigenvalues} holds. Moreover, we generate a sequence of graphs of size $N$ from the graphon $M$, by connecting nodes $i$ and $j$ with weight $M(\frac{i-1}{N},\frac{j-1}{N})$, for $i,j\in\{1,\ldots,N\}$. Similar to the deduction in Example \ref{graphs}, we can show that this sequence of graphs and the graphon $M$ satisfy Assumption \ref{approximation}.

\end{example}

We note that graphons in Example \ref{graphs} and \ref{secondgraphs} are both of the form $M(\alpha,\beta)=\sum_{i=1}^m g_i(\alpha)h_i(\beta)$, which are a class of finite rank graphons (see \cite[Section 4.2]{Medina} for more details). As pointed out by Lov\'{a}sz (see \cite[Exercise 14.56]{Lovasz}), if $M$ and $G$ are two graphons with finite rank, the graphons (if they are still graphons) $M+G$, $MG$ and $M\otimes G$ have finite rank; if $M$ has finite rank and $G$ is arbitrary, then the graphon (if it's still a graphon) $M\circ G+G\circ M$ has finite rank.

\section{An application example in network security}
\qquad In this section, we apply the theoretical results to a network security example, where all users are facing potential cybersecurity threats.

The users are distributed in a graph, in which each node represents a group composed of a cluster of users (firms, for example), and the edges denote the connections among groups. It's noted that security in both wireless and wireline communication networks relies not only on the security investments implemented by individual users, but also on the interdependency among devices of users \cite{Zhang2021}. Therefore, based on the model in \cite[Section 4]{minli}, we assume the network safety level of user $\mathcal{A}_i$ and the performance criteria of user $\mathcal{A}_i$ are characterized by our (\ref{1})-(\ref{2}).
Here, $x^i$ denotes the network safety level of user $\mathcal{A}_i$ in the $q$th group; the control input $u^i$ represents the investment of $\mathcal{A}_i$ in cybersecurity such as purchasing anti-virus software and repairing the device in virus or Trojan; $W^0$ is the common noise which can be interpreted as the persistent signal disturbance inside the network \cite{Zhang2021}. We refer the reader to \cite{Siwe,Vasal} for more details of modelling network security by MFG and GMFG theory.

To obtain explicitly the Nash equilibrium, we set $Ex_0^i=\mu_q=1$ for all $1\leq i\leq K$, $\mu(\alpha)=1$, $A=D=\eta=\Sigma_0=1$, $B=R=H=2$, $Q=Q_T=\frac{3}{2}$. Let $M(\alpha,\beta)=-\cos(2\pi(\alpha-\beta))$ be the limit graphon of $M^{[N]}$, i.e., $M^{[N]}$ and $M$ satisfy Assumption \ref{approximation} (the rationality of this setting is validated in Example \ref{graphs}). We note that these given parameters satisfy Assumptions \ref{finiteeigenvalues}, \ref{monotonicity}, \ref{assumptionmu} and \ref{approximation}.

In this case, the Riccati equation in (\ref{optimalx}) reads
\begin{equation*}
  \dot{f}_t-f_t^2+2f_t+3=0,\ f_T=3,
\end{equation*}
and $f_t\equiv 3$ is its unique solution. The FBSDEs in (\ref{FBSDEl}) yield for $l=1,2$
\begin{equation}\label{FBSDEapplication}
\left\{
  \begin{split}
    & z_t^l=\int_0^t(-\frac{5}{2}z_s^l+\frac{1}{2}g_s^l)ds\\
    & g_t^l=-6z_T^l+\int_t^T(-2g_s^l-3z_s^l)ds-\int_t^T q_s^{1l}dW_s^0.
  \end{split}
  \right.
\end{equation}
Obviously, one has $z^1=z^2$, $g^1=g^2$ and $q^{1l}=0$ for $l=1,2$. On the other hand, the solution of backward ODE (\ref{mathringg}) is given by
\begin{equation*}
  \mathring{g}_t=-3(e^{2(t-T)}+1).
\end{equation*}
Then, the strategy (\ref{epsilonnashstrategy}) reads
\begin{align}\label{controlapplication}
  u_t^{oi} & =-\frac{3}{2}x_t^{oi}-\frac{1}{2}\bar{g}_t^q\nonumber\\
  & =-\frac{3}{2}x_t^{oi}-\frac{1}{2}\Big[\sum_{l=1}^2 g_t^l N\int_{P_q}f_l(\alpha)d\alpha+\mathring{g}_t N\int_{P_q}\Big(\mathds{1}-\sum_{l=1}^2\langle \mathds{1},f_l\rangle f_l\Big)(\alpha)d\alpha\Big]\nonumber\\
  & =-\frac{3}{2}x_t^{oi}-\frac{1}{2}g_t^1 N\frac{\sqrt{2}}{2\pi}\Big(\sin(2\pi\frac{q}{N})-\sin(2\pi\frac{q-1}{N})\Big)\nonumber\\
  &\quad +\frac{1}{2}g_t^2 N\frac{\sqrt{2}}{2\pi}\Big(\cos(2\pi\frac{q}{N})-\cos(2\pi\frac{q-1}{N})\Big)+\frac{3}{2}(e^{2(t-T)}+1)\nonumber\\
  & =-\frac{3}{2}x_t^{oi}+\frac{\sqrt{2}}{4\pi}N\Big[\Big(\cos(2\pi\frac{q}{N})-\cos(2\pi\frac{q-1}{N})\Big)-\Big(\sin(2\pi\frac{q}{N})-\sin(2\pi\frac{q-1}{N})\Big)\Big]g_t^1+\frac{3}{2}(e^{2(t-T)}+1)\nonumber\\
  & =-\frac{3}{2}x_t^{oi}+\frac{\sqrt{2}}{4\pi}N\Big[-2\sin(\pi\frac{2q-1}{N})\sin(\pi\frac{1}{N})-2\cos(\pi\frac{2q-1}{N})\sin(\pi\frac{1}{N})\Big]g_t^1+\frac{3}{2}(e^{2(t-T)}+1)\nonumber\\
  & =-\frac{3}{2}x_t^{oi}-\frac{N}{\pi}\sin(\pi\frac{1}{N})\sin(\pi(\frac{2q-1}{N}+\frac{1}{4}))g_t^1+\frac{3}{2}(e^{2(t-T)}+1),
\end{align}
where $g^1$ is the solution of (\ref{FBSDEapplication}). For user $\mathcal{A}_i$, $u^{oi}$ given by (\ref{controlapplication}) is an $\epsilon$-Nash equilibrium of the network security example.

\section{Conclusions and future work}
\qquad This article is concerned with the study of the linear quadratic graphon mean field games where the individual agents are subject to the idiosyncratic noises and common noise. The optimal strategies for the limit LQ-GMFG problem are derived through the consistency condition, and the $\epsilon$-Nash equilibrum for the finite large population games is established.

It is worth noticing that Tchuendom \cite{Foguen-Tchuendom1} obtained some uniqueness results for linear-quadratic mean field games with common noise. We would like to mention that the uniqueness of the solution for the limit LQ-GMFG with common noise has not been obtained in the present paper. Thus, it would be interesting to show the uniqueness of the solution for the limit LQ-GMFG. In addition, as pointed out by Bensoussan et al. \cite{Bensoussan1}, it is important to consider the large population games with partial observations. Thus, to consider graphon mean field games with partial observation would be another meaningful direction. We leave these as our future work.

\section*{Acknowledgement}
\qquad The authors would like to thank Professors Alexander Aurell,  Peter E. Caines, Zechun Hu, Jianhui Huang, Qi L\"{u}, Rinel Foguen Tchuendom, Bing-Chang Wang, Ruimin Xu, Sheung Chi Phillip Yam, Juliang Yin and Jiayang Yu for helpful suggestions and discussions.

\section*{Appendix}
\textbf{The proof of Proposition \ref{graphonunique}}
\begin{proof}
Following the proofs of Proposition 4.3 and Lemma A.2 of \cite{Caines}, consider any given measurable sets $\mathcal{S},\ \mathcal{T}\subset [0,1]$ and arbitrary $\varepsilon>0$. Denoting $\lambda$ be the Lebesgue measure on $\mathbb{R}$, there exist open sets $\mathcal{S}\subset\mathcal{S}^o$ and $\mathcal{T}\subset\mathcal{T}^o$ such that $\lambda(\mathcal{S}^o\setminus\mathcal{S})\leq\varepsilon$, $\lambda(\mathcal{T}^o\setminus\mathcal{T})\leq\varepsilon$. Letting $\mathcal{S}_1^o=\mathcal{S}^o\bigcap(0,1)$ and $\mathcal{T}_1^o=\mathcal{T}^o\bigcap(0,1)$, we can find a finite integer $s^*$ and constituent disjoint open intervals $I_i^{\mathcal{S}}\subset[0,1]$, $1\leq i\leq s^*$, such that $U_{s^*}\triangleq\bigcup_{i=1}^{s^*}I_i^{\mathcal{S}}\subset\mathcal{S}_1^o$ and $\lambda(\mathcal{S}_1^o\setminus U_{s^*})\leq\varepsilon$. Similarly, we can find a finite integer $t^*$ and disjoint open intervals $I_i^{\mathcal{T}}\subset[0,1]$, $1\leq i\leq s^*$, such that $U_{t^*}\triangleq\bigcup_{j=1}^{t^*}I_j^{\mathcal{T}}\subset\mathcal{T}_1^o$ and $\lambda(\mathcal{T}_1^o\setminus U_{t^*})\leq\varepsilon$. We note that $(s^*,t^*)$ depends on $(\mathcal{S},\mathcal{T},\varepsilon)$.

Let $\bigtriangleup$ denote the symmetric difference. Thus, we have $\lambda(\mathcal{S}\bigtriangleup U_{s^*})\leq 2\varepsilon$, $\lambda(\mathcal{T}\bigtriangleup U_{t^*})\leq 2\varepsilon$, which deduce $\lambda\times\lambda((\mathcal{S}\times\mathcal{T})\bigtriangleup(U_{s^*}\times U_{t^*}))\leq 2\varepsilon^2+4\varepsilon<6\varepsilon$. Since $|(M^{[N]}-M)(\alpha,\beta)|\leq 2$ for all $\alpha,\beta\in [0,1]$, we have
\begin{equation}\label{StimesT}
  \left|\int_{\mathcal{S}\times\mathcal{T}}(M^{[N]}-M)(\alpha,\beta)d\alpha d\beta-\int_{U_{s^*}\times U_{t^*}}(M^{[N]}-M)(\alpha,\beta)d\alpha d\beta\right|<12\varepsilon.
\end{equation}
We take a sufficiently large $N_0$, depending on $s^*$ and $t^*$ (and so on $(\mathcal{S},\mathcal{T},\varepsilon)$), such that for $N\geq N_0$, $s^*/N\leq\varepsilon$ and $t^*/N\leq\varepsilon$. Consider $N\geq N_0$ with the $N$-uniform partition $\{P_1,\ldots,P_N\}$ of $[0,1]$, we select some subintervals from $\{P_1,\ldots,P_N\}$ whenever their interiors are contained in $U_{s^*}$. The selected subcollection denoted by $\{P_{i_r}\}_{r=1}^{r_N}$. Similarly, select a subcollection $\{P_{j_\tau}\}_{\tau=1}^{\tau_N}$. Denote $\hat{U}_{s^*}=\bigcup_{r=1}^{r_N}P_{i_r}$ and $\hat{U}_{t^*}=\bigcup_{\tau=1}^{\tau_N}P_{j_\tau}$. Follow from \cite{Caines}, one has
$$
\lambda(U_{s^*}\setminus\hat{U}_{s^*})\leq 2s^*/N\leq 2\varepsilon,\quad \lambda(U_{t^*}\setminus\hat{U}_{t^*})\leq 2t^*/N\leq 2\varepsilon.
$$
Then for all $N\geq N_0$, we have
\begin{equation}\label{Us*Ut*}
\begin{split}
   &\left|\int_{U_{s^*}\times U_{t^*}}(M^{[N]}-M)(\alpha,\beta)d\alpha d\beta-\int_{\hat{U}_{s^*}\times U_{t^*}}(M^{[N]}-M)(\alpha,\beta)d\alpha d\beta\right|\leq 4\varepsilon,\\
   &\left|\int_{\hat{U}_{s^*}\times U_{t^*}}(M^{[N]}-M)(\alpha,\beta)d\alpha d\beta-\int_{\hat{U}_{s^*}\times \hat{U}_{t^*}}(M^{[N]}-M)(\alpha,\beta)d\alpha d\beta\right|\leq 4\varepsilon.
\end{split}
\end{equation}
Combining (\ref{StimesT}) and (\ref{Us*Ut*}), by the triangle inequality, one has that for all $N\geq N_0$,
\begin{equation}\label{combining}
  \left|\int_{\mathcal{S}\times\mathcal{T}}(M^{[N]}-M)(\alpha,\beta)d\alpha d\beta-\int_{\hat{U}_{s^*}\times \hat{U}_{t^*}}(M^{[N]}-M)(\alpha,\beta)d\alpha d\beta\right|\leq 20\varepsilon.
\end{equation}
By the definition of $\hat{U}_{s^*}$ and $\hat{U}_{t^*}$, we obtain
\begin{align}\label{hatUs*hatUt*}
   \left|\int_{\hat{U}_{s^*}\times \hat{U}_{t^*}}(M^{[N]}-M)(\alpha,\beta)d\alpha d\beta\right|=&\left|\int_{\hat{U}_{s^*}}\int_{\hat{U}_{t^*}}(M^{[N]}-M)(\alpha,\beta)d\beta d\alpha\right|\nonumber\\
   =&\left|\sum_{r=1}^{r_N}\sum_{\tau=1}^{\tau_N}\int_{P_{i_r}}\int_{P_{j_\tau}}(M^{[N]}-M)(\alpha,\beta)d\beta d\alpha\right|\nonumber\\
   \leq & \sum_{r=1}^{r_N}\sum_{\tau=1}^{\tau_N}\left|\int_{P_{i_r}}\int_{P_{j_\tau}}(M^{[N]}-M)(\alpha,\beta)d\beta d\alpha\right|\nonumber\\
   \leq & \sum_{i=1}^{N}\sum_{j=1}^{N}\left|\int_{P_{i}}\int_{P_{j}}(M^{[N]}-M)(\alpha,\beta)d\beta d\alpha\right|\nonumber\\
   \leq & N\cdot\max_{1\leq i\leq N}\sum_{j=1}^{N}\left|\int_{P_{i}}\int_{P_{j}}(M^{[N]}-M)(\alpha,\beta)d\beta d\alpha\right|\nonumber\\
   \leq & N\cdot\max_{1\leq i\leq N}\int_0^1\left|\int_{P_{i}}(M^{[N]}-M)(\alpha,\beta)d\alpha\right|d\beta\nonumber\\
   = & N\cdot\max_{1\leq i\leq N}\|(M^{[N]}-M)\mathds{1}_{P_i}\|_1=E_N.
\end{align}
Then by (\ref{combining}) and (\ref{hatUs*hatUt*}), for all $N\geq N_0$ depending on $(\mathcal{S},\mathcal{T},\varepsilon)$, we have
\begin{equation}\label{EN+20varepsilon}
  \left|\int_{\mathcal{S}\times\mathcal{T}}(M^{[N]}-M)(\alpha,\beta)d\alpha d\beta\right|\leq \left|\int_{\hat{U}_{s^*}\times \hat{U}_{t^*}}(M^{[N]}-M)(\alpha,\beta)d\alpha d\beta\right|+20\varepsilon\leq E_N+20\varepsilon.
\end{equation}
Then, for any given measurable sets $\mathcal{S},\ \mathcal{T}\subset [0,1]$,
\begin{equation*}
  \lim_{N\rightarrow\infty}\left|\int_{\mathcal{S}\times\mathcal{T}}(M^{[N]}-M)(\alpha,\beta)d\alpha d\beta\right|=0.
\end{equation*}

Suppose there is another graphon limit $\hat{M}$ satisfying Assumption \ref{approximation}. For any $\mathcal{S}\times\mathcal{T}\subset[0,1]^2$ and any $\delta>0$,  there exists $N_0$ (depending on $(\mathcal{S},\mathcal{T},\delta)$) such that for $N\geq N_0$,
$$
\left|\int_{\mathcal{S}\times\mathcal{T}}(M^{[N]}-M)(\alpha,\beta)d\alpha d\beta\right|\leq\delta,\quad \left|\int_{\mathcal{S}\times\mathcal{T}}(M^{[N]}-\hat{M})(\alpha,\beta)d\alpha d\beta\right|\leq\delta.
$$
Then for any $\mathcal{S}\times\mathcal{T}\subset[0,1]^2$ and any $\delta>0$, one has
$$
\left|\int_{\mathcal{S}\times\mathcal{T}}(M-\hat{M})(\alpha,\beta)d\alpha d\beta\right|\leq 2\delta,
$$
which implies for any $\mathcal{S}\times\mathcal{T}\subset[0,1]^2$,
$$
\left|\int_{\mathcal{S}\times\mathcal{T}}(M-\hat{M})(\alpha,\beta)d\alpha d\beta\right|=0.
$$
Thus, by the definition of cut norm, we obtain $\|M-\hat{M}\|_\Box=0$.
\end{proof}


\begin{thebibliography}{99}
\bibitem{Huang}M. Huang, R.P. Malham\'{e}, and P.E. Caines, Large population stochastic dynamic games: closed-loop Mckean-Vlasov systems and the Nash certainty equivalence principle, Communications in Information and Systems, 6(3):221-252, 2006.
\bibitem{Huang1}M. Huang, P.E. Caines, and R. P. Malham\'{e}, Large-population cost-coupled LQG problems with nonuniform agents: individual-mass behavior and decentralized $\varepsilon$-Nash equilibria, IEEE Transactions on Automatic Control, 52(9):1560-1571, 2007.
\bibitem{Lions}J.M. Lasry and P.L. Lions, Mean field games, Japanese Journal of Mathematics, 2(1):229-260, 2007.
\bibitem{Carmona}R. Carmona and F. Delarue, Probabilistic analysis of mean-field games, SIAM Journal on Control and Optimization, 51(4):2705-2734, 2013.
\bibitem{Carmona1}R. Carmona and F. Delarue, Probabilistic Theory of Mean Field Games with Applications I, Springer, Cham, Switzerland, 2018.
\bibitem{Carmona2}R. Carmona and F. Delarue, Probabilistic Theory of Mean Field Games with Applications II, Springer, Cham, Switzerland, 2018.
\bibitem{Bensoussan}A. Bensoussan, K.C.J. Sung, S.C.P. Yam and S.P. Yung, Linear-quadratic mean field games, Journal of Optimization Theory and Applications, 169:496-529, 2016.
\bibitem{Ma}Y. Ma and M. Huang, Linear quadratic mean field games with a major player: The multi-scale approach, Automatica, 113, 108774, 2020.
\bibitem{Xu}R. Xu and F. Zhang, $\epsilon$-Nash mean-field games for general linear-quadratic systems with applications, Automatica, 114, 108835, 2020.
\bibitem{Delarue17} F. Delarue, Mean field games: A toy model on an Erd\"{o}s-Renyi graph, ESAIM: Proceedings and Surveys, 60:1-26, 2017.
\bibitem{Lovasz}L. Lov\'{a}sz, Large Networks and Graph Limits, vol. 60, American Mathematical Society, Providence, RI, 2012.
\bibitem{Lovasz1}L. Lov\'{a}sz and B. Szegedy, Limits of dense graph sequences, Journal of Combinatorial Theory, Series B, 96(6):933-957, 2006.
\bibitem{Borgs}C. Borgs, J.T. Chayes, L. Lov\'{a}sz, V.T. S\'{o}s, and K. Vesztergombi, Convergent sequences of dense graphs i: Subgraph frequencies, metric properties and testing, Advances in Mathematics, 219(6):1801-1851, 2008.
\bibitem{Borgs1}C. Borgs, J.T. Chayes, L. Lov\'{a}sz, V.T. S\'{o}s, and K. Vesztergombi, Convergent sequences of dense graphs ii. multiway cuts and statistical physics, Annals of Mathematics, 176(1):151-219, 2012.

\bibitem{Gao1}S. Gao and P.E. Caines, Graphon control of large-scale networks of linear systems, IEEE Transactions on Automatic Control, 65(10):4090-4105, 2019.
\bibitem{Gao2}S. Gao and P.E. Caines, Graphon linear quadratic regulation of large-scale networks of linear systems, In 2018 IEEE Conference on Decision and Control (CDC) (pp. 5892-5897). IEEE, 2018.
\bibitem{Gao3}S. Gao and P.E. Caines, Optimal and approximate solutions to linear quadratic regulation of a class of graphon dynamical systems, In 2019 IEEE 58th Conference on Decision and Control (CDC) (pp. 8359-8365). IEEE, 2019.
\bibitem{Gao4}S. Gao and P.E. Caines, Spectral representations of graphons in very large network systems control, In 2019 IEEE 58th conference on decision and Control (CDC) (pp. 5068-5075). IEEE, 2019.
\bibitem{Carmona3}R. Carmona, D.B. Cooney, C.V. Graves, and M. Lauriere, Stochastic graphon games: I. the static case, Mathematics of Operations Research, 47(1):750-778, 2022.
\bibitem{Parise}F. Parise and A. Ozdaglar, Graphon games: a statistical framework for network games and interventions, Econometrica, 91(1):191-225, 2023.

\bibitem{Caines1}P.E. Caines and M. Huang, Graphon mean field games and the GMFG equations, in Proceedings of the 57th IEEE Conference on Decision and Control (CDC) (pp.4129-4134). IEEE, 2018.
\bibitem{Caines2}P.E. Caines and M. Huang, Graphon mean field games and the GMFG equations: $\epsilon$-Nash equilibria, in Proceedings of the 58th IEEE Conference on Decision and Control (CDC) (pp. 286-292). IEEE, 2019.
\bibitem{Caines}P.E. Caines and M. Huang, Graphon mean field games and their equations, SIAM Journal on Control and Optimization, 59(6):4373-4399, 2021.
\bibitem{Gao}S. Gao, R.F. Tchuendom, and P.E. Caines, Linear quadratic graphon field games, arXiv preprint, arXiv:2006.03964, 2020.
\bibitem{Gao5}S. Gao, P.E. Caines, and M. Huang, LQG graphon mean field games: Analysis via graphon invariant subspaces, IEEE Transactions on Automatic Control, 68(12):7482-7497, 2023.
\bibitem{Aurell}A. Aurell, R. Carmona, and M. Lauri\`{e}re, Stochastic graphon games: II. the linear-quadratic case, Applied Mathematics and Optimization, 85(3):39, 2022.
\bibitem{Foguen-Tchuendom}R.F. Tchuendom, S. Gao, M. Huang and P.E. Caines, Optimal network location in infinite horizon LQG graphon mean field games, In 2022 IEEE 61st Conference on Decision and Control (CDC) (pp. 5558-5565). IEEE, 2022.
\bibitem{Amini}H. Amini, Z. Cao and A. Sulem, Stochastic graphon mean field games with jumps and approximate Nash equilibria, arXiv preprint, arXiv:2304.04112, 2023.
\bibitem{Bayraktar}E. Bayraktar, S. Chakraborty and R. Wu, Graphon mean field systems, The Annals of Applied Probability, 33(5):3587-3619, 2023.
\bibitem{Bayraktar1}E. Bayraktar, R. Wu, and X. Zhang, Propagation of chaos of forward-backward stochastic differential equations with graphon interactions, Applied Mathematics and Optimization, 88, 25, 2023.
\bibitem{Amini1}H. Amini, Z. Cao and A. Sulem, Graphon mean-field backward stochastic differential equations with jumps and associated dynamic risk measures, SSRN.4162616, 2023.
\bibitem{Aurell1}A. Aurell, R. Carmona, G. Dayan{\i}kl{\i} and M. Lauri\`{e}re, Finite state graphon games with applications to epidemics, Dynamic Games and Applications, 12(1):49-81, 2022.
\bibitem{Carmona4} R. Carmona, F. Delarue and D. Lacker, Mean field games with common noise, The Annals of Probability, 44(6):3740-3803, 2016.
\bibitem{Foguen-Tchuendom1}R.F. Tchuendom, Uniqueness for linear-quadratic mean field games with common noise, Dynamic Games and Applications, 8:199-210, 2018.
\bibitem{Carmona5}R. Carmona, J.P. Fouque and L.H. Sun, Mean field games and systemic risk, Communications in Mathematical Sciences, 13:911-933, 2015.
\bibitem{Bensoussan1}A. Bensoussan, X. Feng and J. Huang, Linear-quadratic-Gaussian mean-field-game with partial observation and common noise, Mathematical Control and Related Fields, 11(1):23-46, 2021.
\bibitem{Graber}P.J. Graber, Linear quadratic mean field type control and mean field games with common noise, with application to production of an exhaustible resource, Applied Mathematics and Optimization, 74:459-486, 2016.
\bibitem{Wang}B.C. Wang, H. Zhang, and J.F. Zhang, Linear quadratic mean field social control with common noise: A directly decoupling method, Automatica, 146, 110619, 2022.
\bibitem{Hua}T.J. Hua and P. Luo, Linear-quadratic extended mean field games with common noises, arXiv preprint arXiv:2311.04001, 2023.
\bibitem{Dunyak}A. Dunyak and P.E. Caines, Graphon Field Tracking Games with Discrete Time Q-noise, In 2023 IEEE 62nd Conference on Decision and Control (CDC) (pp.8194-8199). IEEE, 2023.
\bibitem{Tangpi}L. Tangpi and X. Zhou, Optimal investment in a large population of competitive and heterogeneous agents, Finance and Stochastics, 28(2):497-551, 2024.
\bibitem{Janson}S. Janson, Graphons, cut norm and distance, couplings and rearrangements, arXiv preprint arXiv:1009.2376, 2010.
\bibitem{lvqi} Q. L\"{u} and X. Zhang, Mathematical Control Theory for Stochastic Partial Differential Equations, Springer, Berlin, 2021.
\bibitem{Freiling}G. Freiling, A survey of nonsymmetric Riccati equations, Linear algebra and its applications, 351:243-270, 2002.
\bibitem{Schilling} R.L. Schilling, Measures, integrals and martingales, Cambridge University Press, 2017.
\bibitem{Brooks}R.A. Brooks, Conditional expectations associated with stochastic processes, Pacific Journal of Mathematics, 41(1):33-42, 1972.
\bibitem{Royden}H.L. Royden and P. Fitzpatrick, Real Analysis (Vol. 2), New York: Macmillan, 1968.
\bibitem{Dunyak1}A. Dunyak and P.E. Caines, Quadratic Optimal Control of Graphon Q-noise Linear Systems, arXiv preprint arXiv:2407.00212, 2024.
\bibitem{Dunyak2}A. Dunyak and P.E. Caines, Linear Stochastic Processes on Networks and Low Rank Graph Limits, In International Conference on Complex Networks and Their Applications (pp. 395-407), Cham: Springer Nature Switzerland, 2023.
\bibitem{Xu1}R. Xu and J. Shi, $\varepsilon$-Nash mean-field games for linear-quadratic systems with random jumps and applications, International Journal of Control, 94(5):1415-1425, 2021.
\bibitem{Peng}S. Peng and Z. Wu, Fully coupled forward-backward stochastic differential equations and applications to optimal control, SIAM Journal on Control and Optimization, 37(3):825-843, 1999.
\bibitem{Carmona6}R. Carmona, F. Delarue, Mean field forward-backward stochastic differential equations, Electron. Commun. Probab, 18:1-15, 2013.
\bibitem{Salhab}R. Salhab, R.P. Malham\'{e} and J. Le Ny, Collective stochastic discrete choice problems: A min-LQG dynamic game formulation, IEEE Transactions on Automatic Control, 65(8):3302-3316, 2019.
\bibitem{Perko}L. Perko, Differential Equations and Dynamical Systems, Springer, Berlin, 2013.

\bibitem{Revuz}D. Revuz and M. Yor, Continuous martingales and Brownian motion (Vol. 293), Springer Science and Business Media, 2013.
\bibitem{Karatzas}I. Karatzas and S. Shreve, Brownian motion and stochastic calculus (Vol. 113), springer, 2014.
\bibitem{Medina}M. Avella-Medina, F. Parise, M.T. Schaub and S. Segarra, Centrality measures for graphons: Accounting for uncertainty in networks, IEEE Transactions on Network Science and Engineering, 7(1):520-537, 2018.
\bibitem{Zhang2021}W. Zhang and C. Peng, Indefinite mean-field stochastic cooperative linear-quadratic dynamic difference game with its application to the network security model, IEEE Transactions on Cybernetics, 52(11):11805-11818, 2021.
\bibitem{minli}M. Li, T. Nie, S. Wang and K. Yan, Incomplete Information Mean-Field Games and Related Riccati Equations, Journal of Optimization Theory and Applications, 1-22, 2024.
\bibitem{Siwe}A.T. Siwe and H. Tembine, Network security as public good: A mean-field-type game theory approach, In 13th International Multi-Conference on Systems, Signals \& Devices (SSD) (pp. 601-606). IEEE, 2016.
\bibitem{Vasal}D. Vasal, R. Mishra and S. Vishwanath, Sequential decomposition of graphon mean field games, In 2021 American Control Conference (ACC) (pp. 730-736). IEEE, 2021.


\end{thebibliography}
\end{document}